\documentclass[a4paper]{amsart}
\usepackage[utf8]{inputenc}
\usepackage{amsmath,amssymb,mathrsfs,amsfonts,mathtools}
\usepackage{mathtools}
\usepackage{graphicx,geometry}
\usepackage{float,comment,xcolor}
\usepackage{enumitem}
\usepackage{tikz-cd}
\newtheorem*{claim}{Claim}
\newtheorem{theorem}{Theorem}[section]
\newtheorem{lemma}[theorem]{Lemma}

\newtheorem{prop}[theorem]{Proposition}
\newtheorem{cor}[theorem]{Corollary}

\theoremstyle{definition}
\newtheorem{definition}[theorem]{Definition}
\theoremstyle{remark}
\newtheorem{remark}[theorem]{Remark}
\numberwithin{equation}{section}

\title{A note On the existence of solutions to Hitchin's self-duality equations}
\date{\today}
\author{Yu Feng }
\address{Chern Institute of Mathematics and LPMC, Nankai University \newline \indent Tianjin 300071, China,}
\email{yuf@nankai.edu.cn}
\author{Shuo Wang$^\dagger$}
\address{School of Mathematical  Sciences, University of Science and Technology of China, Hefei, 230026, People's Republic of China}
\email{ws220122@mail.ustc.edu.cn}
\thanks{ B.X. is supported in part by the Project of Stable Support for Youth Team in Basic Research Field, CAS (Grant No. YSBR-001) and NSFC (Grant No. 12271495).  }
\thanks{$^\dagger$S.W. is the corresponding author.}
\author{Bin Xu}
\address{School of Mathematical Sciences, University of Science and Technology of China, Hefei, 230026, People's Republic of China}
\email{bxu@ustc.edu.cn}
\begin{document}
\maketitle

\begin{abstract}
In 1987, Hitchin introduced the self-duality equations on rank-2 complex vector bundles over compact Riemann surfaces with genus greater than one as a reduction of the Yang-Mills equation and established the existence of solutions to these equations starting from a Higgs stable bundle. In this paper, we fill in some technical details in Hitchin's original proof by the following three steps. First, we reduce the existence of a solution of class $L_1^2$ to minimizing the energy functional within a Higgs stable orbit of the $L_2^2$ complex gauge group action. Second, using this transformation, we obtain a solution of class $L_1^2$ in this orbit. These two steps primarily follow Hitchin's original approach. Finally, using the Coulomb gauge, we construct a smooth solution by applying an $L_2^2$ unitary gauge  transformation to the $L_1^2$ solution constructed previously. This last step provides additional technical details to Hitchin's original proof.
\end{abstract}

\section{Introduction}

In his 1987 paper~\cite{hitchin1987self}, Hitchin introduced the self-duality equations for rank-2 Hermitian bundles over compact Riemann surfaces of genus greater than one. Employing Sobolev completion and weak convergence techniques, he obtained a solution to the self-duality equations via a gauge transformation applied to a stable Higgs bundle, as elaborated in~\cite[Theorem 4.3]{hitchin1987self}. This result stands as a foundation in the study of Hitchin systems and provides a geometric framework of non-Abelian Hodge correspondence.

The space of smooth complex Higgs structures on Hermitian bundles forms an infinite-dimensional affine space equipped with a natural Kähler structure, where Hitchin derived his self-duality equations. Hitchin computed the moment map for the action of unitary gauge group on this space, constructed an energy functional, and reduce the existence of a solution of class $L_1^2$ to finding the minimum of this energy functional on a Higgs stable orbit, thereby proving the existence theorem. Following Hitchin's approach, our proof proceeds in three steps. First, using Hitchin's energy functional, we extend his reduction process on the space of smooth complex Higgs structures to the space of $L_1^2$ complex Higgs structures. Specifically, we translate the existence of a solution of class $L_1^2$ into finding a minimizer of the energy functional restricted to the Higgs stable orbit under the action of bundle automorphisms of class $L_2^2$. Second, by applying $L_2^2$ bundle automorphisms to a stable Higgs bundle, we construct an $L_1^2$ solution to the self-duality equations. An important observation here is that a unitary bundle automorphism of class $L_2^2$ maps one $L_1^2$ solution to another $L_1^2$ solution to the self-duality equations. Finally, employing the Coulomb gauge, we construct a smooth solution to the self-duality equations by applying a unitary bundle automorphism of class $L_2^2$ to the solution obtained in the second step, as stated in Theorem~\ref{main2}. This final step provides more details for the original proof by Hitchin and rigorously confirms the existence of smooth solutions to the self-duality equations as originally demonstrated in~\cite[Theorem~4.3]{hitchin1987self}.

We conclude this section by outlining the organization of the remaining four sections of this note. In Section~2, we provide the necessary background on the self-duality equations and gauge groups, and state the main result, Theorem~\ref{main}. In Section~3, we transform the existence problem of $L_1^2$ solutions to the self-duality equations into finding the critical point of the energy functional restricted to the orbit under the $L_2^2$ complex gauge group action. Building on this crucial transformation, we construct a solution lying in $L_1^2$ space to the self-duality equations in Section~4 and establish Theorem~\ref{main1}. Finally, in Section~5, we demonstrate how to construct a smooth solution from the solution obtained in Section~4 and establish Theorem~\ref{main2}. Together, Theorems~\ref{main1} and~\ref{main2} provide a complete proof of the main result, Theorem~\ref{main}.

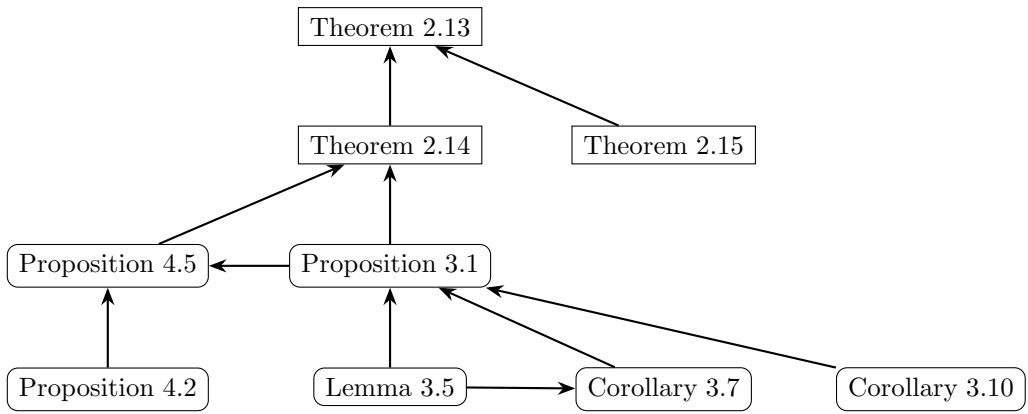
\vspace{1cm}

\begin{center}
\begin{tikzcd}[row sep=3em, column sep=3em]
  &  |[draw, rectangle]| \text{Theorem \ref{main}}  & &\\ 
  & |[draw, rectangle, thin]| \text{Theorem \ref{main1}}\arrow[u,-Stealth, thick]  & |[draw, rectangle]| \text{Theorem \ref{main2}}\arrow[ul,-Stealth, thick] & \\
  |[draw, rectangle, rounded corners]| \text{Proposition \ref{generic-case}} \arrow[ur,-Stealth, thick]& |[draw, rectangle, rounded corners]| \text{Proposition \ref{eq-conditions}}\arrow[u,-Stealth, thick] \arrow[l,-Stealth, thick] & & \\
  |[draw, rectangle, rounded corners]| \text{Proposition \ref{w-converge}}\arrow[u,-Stealth, thick] & |[draw, rectangle, rounded corners]| \text{Lemma \ref{grad}}\arrow[u,-Stealth, thick]  \arrow[r,-Stealth, thick]& |[draw, rectangle, rounded corners]| \text{Corollary \ref{cp-reg}}\arrow[ul,-Stealth, thick]& |[draw, rectangle, rounded corners]| \text{Corollary \ref{s=0}}\arrow[ull,-Stealth, thick] 
\end{tikzcd}
\end{center}

\vspace{1cm}

\section{Background and Main Result}

Let $M$ be a compact Riemann surface of genus greater than one, and let $(V, H)$ be a Hermitian bundle over $M$ of rank 2. Denote by $\omega_M$ the Kähler form, normalized such that $\operatorname{Vol}(M, \omega_M) = 1$. In the first three subsections, we introduce the space $N$ of smooth complex Higgs structures on $V$, the group $\mathscr{G}$ of smooth unitary gauge transformations, the group $\mathscr{G}^\mathbb{C}$ of smooth complex gauge transformations, and their respective group actions on $N$. These spaces can be completed under some Sobolev norms, 
yielding the Kähler Banach space $(N_1^2, g, \mathcal{I})$, the Banach real Lie group $\mathscr{G}_2^2$, and the Banach complex Lie group ${\mathscr{G}^\mathbb{C}}_2^2$. Hitchin's self-duality equations can be extended to $N_1^2$, as stated in~\eqref{hitchin-sd}. We also provide the moment map $\mu$ of the action of $\mathscr{G}_2^2$ on $N_1^2$. Subsequently, in the final two subsections, we present Atiyah-Bott's result, as stated in Theorem~\ref{A-B}, along with the main results of this note, Theorems~\ref{main}, \ref{main1}, and \ref{main2}.

\subsection{Self-duality equations}

\subsubsection{Self-duality equations on principal bundles}

Let $P \to M$ be an $SO(3)$ principal bundle, and denote by $\operatorname{ad} P$ its adjoint vector bundle. The \textit{self-duality equations} for $(A, \Phi)$ are given by  
\begin{equation}
\begin{cases}
F_A + [\Phi, \Phi^*] = 0, \\
\mathrm{d}_A^{\prime\prime} \Phi = 0,
\end{cases}\label{Psd}
\end{equation}
where $A$ is a connection on $P$, $\mathrm{d}_A^{\prime\prime}$ is the holomorphic structure on $\operatorname{ad} P\otimes\mathbb{C}$ induced by $A$, and $\Phi \in \Omega^{1,0}(M, \operatorname{ad} P \otimes \mathbb{C})$. We will transfer the self-duality equations \eqref{Psd} on the principal bundle to their version on the associated vector bundle. First, let us describe the associated vector bundle of $P$. Consider the following short exact sequence 
\[
\begin{tikzcd}
1 \arrow[r] & U(1) \arrow[r] & U(2) \arrow[r, "\eta_0"] \arrow[d, dashed] & SO(3) \arrow[r] & 1 \\
            &                & SU(2) \arrow[ur, "\xi"]
\end{tikzcd}
\]
where $U(2) \dashrightarrow SU(2)$ is a multi-valued map given by multiplying by two suitably chosen constants, and $\xi$ is the universal covering map. By the argument in \cite[Appendix A, p. 373]{lawson1989spin}, the induced long exact sequence is given by:
\[
\cdots \to \mathrm{H}^1(M, U(1)) \to \mathrm{H}^1(M, U(2)) \to \mathrm{H}^1(M, SO(3)) \to 0.
\]
implies that there exists a $U(2)$ principal bundle $\tilde{P}$ and a bundle map $\eta: \tilde{P} \to P$ satisfying $\eta(\tilde{p} \cdot a) = \eta(\tilde{p}) \cdot \eta_0(a)$ for any $\tilde{p} \in \tilde{P}, \, a \in U(2).$ Define $V := \tilde{P} \times_{\operatorname{id}} \mathbb{C}^2$ as a rank 2 complex vector bundle over $M$, where $\operatorname{id}$ is the natural representation of $U(2)$ on $\mathbb{C}^2$. The canonical metric on $\mathbb{C}^2$ induces a Hermitian metric $H$ on $V$. The Hermitian bundle $(V, H)$ is called the \textit{associated bundle} of $P$ in Hitchin's original paper~\cite[Section~2]{hitchin1987self}. An important observation is
\[
\operatorname{ad} P \cong \operatorname{End}_0^{\text{sk}} V, \quad \operatorname{ad} \tilde{P} \cong \operatorname{End}^{\text{sk}} V, \quad \operatorname{ad} P \otimes \mathbb{C} \cong \operatorname{End}_0 V, \quad \operatorname{ad} \tilde{P} \otimes \mathbb{C} \cong \operatorname{End} V,
\]
where $\operatorname{End}_0 V$ and $\operatorname{End}^{\text{sk}} V$ are the bundles of trace-free and skew-Hermitian endomorphisms of $(V, H)$, respectively. Thus, $\Phi \in \Omega^{0,1}(M, \operatorname{ad} P \otimes \mathbb{C})$ can be regarded as a trace-free element in $\Omega^{0,1}(M, \operatorname{ad} \tilde{P} \otimes \mathbb{C})$, denoted by $\tilde \Phi$.

Define $\det \tilde{P} := \tilde P \times_{\det} U(1)$ as a $U(1)$-principal bundle over $M$. Following Hitchin \cite[Section 2]{hitchin1987self}, we fix a connection $A_0$ on $\det \tilde{P}$. Then there exists a one-to-one correspondence between connections $A$ on $P$ and connections $\tilde{A}$ on $\tilde{P}$ with induced connection $A_0$ on $\det\,\tilde{P}$. Under this context, we have
\[
F_{\tilde{A}} = F_A + \left(\frac{1}{2} F_{A_0}\right) \operatorname{id}_V \in \Omega^2(M, \operatorname{ad} \tilde{P}) \cong \Omega^2(M, \operatorname{End}^{\text{sk}} V).
\]
In particular, $F_{\tilde{A}}^\perp = F_A \in \Omega^2(M, \operatorname{End}_0^{\text{sk}} V)$, where $F_{\tilde{A}}^\perp$ denotes the trace-free part of $F_{\tilde{A}}$. Hence, the self-duality equations \eqref{Psd} can be reformulated as equations on $\tilde{P}$, given by
\begin{equation*}
\begin{cases}
F_{\tilde{A}}^\perp + [\tilde{\Phi}, \tilde{\Phi}^*] = 0, \\
\mathrm{d}_{\tilde{A}}^{\prime\prime} \tilde{\Phi} = 0.
\end{cases}
\end{equation*}
In fact, $\tilde{P}$ and $\det \tilde{P}$ are the unitary frame bundles of $(V, H)$ and $(\det V, H_0)$, respectively, where $H_0$ is the induced Hermitian metric of $H$. Therefore, by \cite[Definition~2.1.8]{joyce2000compact}, connections on $\tilde{P}$ correspond one-to-one with unitary connections on $(V, H)$, and connections on $\det \tilde{P}$ correspond one-to-one with unitary connections on $(\det V, H_0)$. To summarize, once a unitary connection $A_0$ on $(\det V, H_0)$ is fixed, a connection $A$ on $P$ induces a unitary connection $\tilde{A}$ on $(V, H)$ whose induced connection on $\det V$ is exactly $A_0$. Furthermore, $\Phi$ is a smooth section of $\operatorname{End}_0 V$. At this point, we can discuss the self-duality equations \eqref{Psd} in the version of the associated
bundle as follows.

\subsubsection{Self-duality equations on associated bundles}
Let $(V, H)$ be a rank-2 Hermitian vector bundle over $M$. Fix a unitary connection $A_0$ on $(\det V, H_0)$, where $H_0$ is the induced Hermitian metric on $\det V$. Denote by $\mathscr{A}$ the space of \textit{smooth unitary connections} on $(V, H)$ that induce $A_0$ on $\det V$. Given a fixed connection $A \in \mathscr{A}$, the space $\mathscr{A}$ can be described as an affine space 
\[
\mathscr{A} = A + \Omega^{1, \text{sk}}(M, \operatorname{End}_0 V),
\]
where $\Omega^{1, \text{sk}}(M, \operatorname{End}_0 V)$ consists of smooth 1-forms $\psi$ on $\operatorname{End}_0 V$ satisfying
\[
\langle \psi(s), t \rangle^{H} + \langle s, \psi(t) \rangle^{H} = 0,  \forall s, t \in \Omega^0(M, V).
\] 
On a local chart $(U, z)$, an element of $\Omega^{1, \text{sk}}(M, \operatorname{End}_0 V)$ can be expressed as
$-\theta^{\dagger_H} \, \mathrm{d}z + \theta \, \mathrm{d}\bar{z}$,
where $\theta \in \Omega^0(U, \operatorname{End}_0 V)$, and $\dagger_H$ is the adjoint operator with respect to $H$.

Define the space of \textit{smooth complex Higgs fields} on $V$ as $\Omega^{1,0}(M, \operatorname{End}_0 V)$. For any smooth Higgs field $\Phi$, we can locally express it on a chart $(U, z)$ as $\Phi|_U = \phi \, \mathrm{d}z$, where $\phi \in \Omega^0(U, \operatorname{End}_0 V)$. We denote by $\Phi^*$ the smooth $(0,1)$-form on $\operatorname{End}_0 V$, given locally by $\Phi^*|_U = \phi^{\dagger_H} \, \mathrm{d}\bar{z}$.

Denote by $N := \mathscr{A} \times \Omega^{1,0}(M, \operatorname{End}_0 V)$ the space of \textit{smooth complex Higgs structures} on $V$. Define an invariant Riemannian metric $\langle \cdot, \cdot \rangle_1$ on $\mathscr{A}$ by  
\[
\langle \dot{A}_1, \dot{A}_2 \rangle_1 = -\mathrm{i} \int_M \dot{A}_1^{1,0} \wedge \dot{A}_2^{0,1} + \dot{A}_2^{1,0} \wedge \dot{A}_1^{0,1}, \quad \forall~~ \dot{A}_i \in T_A \mathscr{A}.
\]  
Similarly, we define an inner product $\langle \cdot, \cdot \rangle_2$ on $\Omega^{1,0}(M, \operatorname{End}_0 V)$ by  
\[
\langle \dot{\Phi}_1, \dot{\Phi}_2 \rangle_2 = \mathrm{i} \int_M \dot{\Phi}_1 \wedge \dot{\Phi}_2^* + \dot{\Phi}_2 \wedge \dot{\Phi}_1^*, \quad \forall~~ \dot{\Phi}_i \in \Omega^{1,0}(M, \operatorname{End}_0 V).
\]  
These inner products induce a Riemannian metric on $N$, denoted by $g$. 
The complex structure on $\mathscr{A}$ is given by $\mathcal{I}_1(\dot{A}) := -\mathrm{i} \dot{A}^{1,0} + \mathrm{i} \dot{A}^{0,1}$, while on $\Omega^{1,0}(M, \operatorname{End}_0 V)$ it is given by $\mathcal{I}_2(\dot{\Phi}) := \mathrm{i} \dot{\Phi}$. Together, these complex structures induce a complex structure on $N$, denoted by $\mathcal{I}$, and the resulting triple $(N, g, \mathcal{I})$ forms a Kähler structure. 

The self-duality equations in the setting of vector bundles is
\[
(A, \Phi) \in N \quad \text{such that} \quad
\begin{cases}
F_A^{\perp} + \big[\Phi, \Phi^*\big] = 0, \\
\mathrm{d}_A^{\prime\prime} \Phi = 0,
\end{cases}
\]
where $F_A^{\perp}$ denotes the trace-free part of $F_A$, and $\big[\Phi, \Phi^*\big] := \big[\phi, \phi^{\dagger_H}\big] \, \mathrm{d}z \wedge \mathrm{d}\bar{z}$ locally on $(U,z)$. A pair $(A, \Phi) \in N$ that satisfies the equation above is called a \textit{smooth solution} of the self-duality equations.

\subsubsection{Sobolev Case}
First, we introduce some results on Sobolev spaces
\begin{lemma}\label{Sob property}
Denote by $L_k^p$ the space of $L_k^p$ functions defined on the unit disk in $\mathbb{C}$. Then
\begin{itemize}

    \item[$(1)$] \rm{(\cite[Theorem~6.7]{di2012hitchhikers})} Assume $u \in L_k^p$ for some $0<k<1$ and $1 < p < \frac{2}{k}$. Let $q > 1$ satisfy $\frac{1}{q} = \frac{1}{p} - \frac{k}{2}$. Then, there exists a constant $C$ depending only on $k, p$ such that
\[
\|u\|_{L^q} \leq C \|u\|_{L_k^p}.
\]
This embedding is denoted by $L_k^p \hookrightarrow L^q$. In particular
\[
L_1^2 \hookrightarrow L_{\frac{p}{p+1}}^2 \hookrightarrow L^p\quad\forall1 \leq p < \infty.
\]
We will refer to this result as the \textit{Sobolev Embedding Theorem}.

    \item[$(2)$] \rm{(\cite[Theorem 6.1]{behzadan2021multiplication})} Let $k_i$ and $1 \leq p_i, p < \infty$ for $(i = 1, 2)$ be real numbers, and let $k$ be a nonnegative integer such that $k_i \geq k$, $k_i - k \geq \frac{2}{p_i} - \frac{2}{p}$, and $k_1 + k_2 - k > \frac{2}{p_1} + \frac{2}{p_2} - \frac{2}{p} \geq 0$. If $u_1 \in L_{k_1}^{p_1}$ and $u_2 \in L_{k_2}^{p_2}$, then there exists a constant $C$ depending only on $k, k_i, p, p_i$ such that 
\[
\|u_1 \cdot u_2\|_{L_k^p} \leq C \|u_1\|_{L_{k_1}^{p_1}} \|u_2\|_{L_{k_2}^{p_2}}.
\] 
This embedding is denoted by $L_{k_1}^{p_1} \times L_{k_2}^{p_2} \hookrightarrow L_k^p$. In particular, $L_2^2 \times L_1^2 \hookrightarrow L_1^2$. We will refer to this result as the \textit{Sobolev Multiplicative Theorem}.

\end{itemize}
Let $E$ be a vector bundle over a compact Riemann surface $M$. Denote by $L_k^p(M, E)$ the vector space of $L_k^p$ sections of $E$. Then
\begin{itemize}
    \item[$(3)$] {\rm (\cite[Chapter~4, Proposition~4.4]{taylorpartial})}$L_k^2(M, E)$ is compactly embedded into $L_{k-\varepsilon}^2(M, E)$ for any $k \in \mathbb{R}, \varepsilon>0$. 
\end{itemize}
\end{lemma}

\begin{cor}\label{SobM}
For any $\frac{1}{2} > \varepsilon > 0$, the following embeddings hold
\[
L_1^2 \times L^2 \hookrightarrow L^{2-\varepsilon}
\quad \text{and} \quad
L_1^2 \times L_1^{2-\varepsilon} \hookrightarrow L_1^{2-2\varepsilon}.
\]
\end{cor}

\begin{proof}
In fact, we only need to prove the following two cases
\[
L_1^2 \times L^{2-\varepsilon} \hookrightarrow L^{2-2\varepsilon} \quad \text{and} \quad L^2 \times L_1^{2-\varepsilon} \hookrightarrow L^{2-\varepsilon}.
\]
First, from (1), we know that $L_1^2$ embeds into $L^p$, where $p$ satisfies $\frac{1}{p} = \frac{1}{2-2\varepsilon} - \frac{1}{2-\varepsilon}.$ Then, by Hölder's inequality, we have $L^p \times L^{2-\varepsilon} \hookrightarrow L^{2-2\varepsilon},$ which completes the proof of $L_1^2 \times L^{2-\varepsilon} \hookrightarrow L^{2-2\varepsilon}$. Second, $L_1^{2-\varepsilon}$ embeds into $L^q$, where $q$ satisfies $\frac{1}{q} + \frac{1}{2} = \frac{1}{2-\varepsilon}.$ Then, by Hölder's inequality, we have $L^2 \times L^q \hookrightarrow L^{2-\varepsilon},$ which completes the proof of $L^2 \times L_1^{2-\varepsilon} \hookrightarrow L^{2-\varepsilon}$. 
\end{proof}

\begin{cor}
$L_1^2(M, E)$ is compactly embedded into $L^p(M, E)$ for any $1 \leq p < \infty$.
\end{cor}

\begin{proof}
By (3), $L_1^2(M, E)$ is compactly embedded into $L_{1-\varepsilon}^2(M, E)$ for any $\varepsilon > 0$. From \cite[]{di2012hitchhikerʼs}, $L_{1-\varepsilon}^2(M, E)$ is embedded into $L^{2/\varepsilon}(M, E)$. Combining these results, $L_1^2(M, E)$ is compactly embedded into $L^{2/\varepsilon}(M, E)$ for each $\varepsilon > 0$. 
\end{proof}

\begin{definition}
The space of \textit{unitary connections of class $L_1^2$} on $(V, H)$ is defined as  
\[
\mathscr{A}_1^2 := \Big\{ A = A^{\text{sm}} + \theta \,\Big|\, \theta \in (L_1^2)^{\text{sk}}(M, \operatorname{End}_0 V \otimes T^*_\mathbb{C} M) \Big\},
\]
where $A^{\text{sm}}$ is a fixed smooth unitary connection in $\mathscr{A}$. Note that the space $\mathscr{A}_1^2$ is independent of the choice of $A^{\text{sm}}$. The space of \textit{complex Higgs fields of class $L_1^2$} is defined as the space of $L_1^2$ sections of the vector bundle $\operatorname{End}_0 V \otimes K_M$, denoted by $L_1^2(M, \operatorname{End}_0 V \otimes K_M)$ or $(L_1^2)^{1,0}(M, \operatorname{End}_0 V)$. We denote $N_1^2 := \mathscr{A}_1^2 \times L_1^2(M, \operatorname{End}_0 V \otimes K_M)$. 
\end{definition}
In this setting, the Kähler structure $(N, g, \mathcal{I})$ extends naturally to $N_1^2$, which we also denote by $(N_1^2, g, \mathcal{I})$. 
\begin{lemma}
$F_A$, $[\Phi, \Phi^*]$, and $\mathrm{d}_A^{\prime\prime} \Phi$ lie in $L^2(M, \operatorname{End}_0 V \otimes \Lambda^2 T_\mathbb{C}^* M)$ for any $(A, \Phi) \in N_1^2$.
\end{lemma}

\begin{proof}
The curvature map $\mathscr{A} \to \Omega^{1,1}(M, \operatorname{End} V)$, which assigns to a connection $A$ its curvature $F_A$, can be continuously extended to:  
\[
\mathscr{A}_1^2 \to L^2(M, \operatorname{End} V \otimes \Lambda^2 T^*_\mathbb{C} M) \quad \text{(\cite[Lemma~1.1]{uhlenbeck1982connections}).}
\]
Locally on $(U, z)$, represent $A = \mathrm{d} - \theta^{\dagger_H} \mathrm{d}z + \theta \mathrm{d}\bar{z}$ and $\Phi = \phi \mathrm{d}z$. Then both $\theta$ and $\phi$ have $L_1^2$ regularity. Locally, we have:
\[
\mathrm{d}_A^{\prime\prime} \Phi = -\left(\frac{\partial \phi}{\partial \bar z} + [\theta, \phi]\right) \mathrm{d}z \wedge \mathrm{d}\bar{z}, \quad 
[\Phi, \Phi^*] = [\phi, \phi^{\dagger_H}] \mathrm{d}z \wedge \mathrm{d}\bar{z}.
\]
By the Sobolev embedding theorem, $L_1^2 \hookrightarrow L^4$, and we know $L_1^2 \times L_1^2 \subseteq L^2$. Therefore, $\frac{\partial \phi}{\partial \bar z}$, $[\phi, \theta]$, and $[\phi, \phi^{\dagger_H}]$ all have $L^2$ regularity. This implies that both $\mathrm{d}_A^{\prime\prime} \Phi$ and $[\Phi, \Phi^*]$ lie in $L^2(M, \operatorname{End}_0 V \otimes \Lambda^2 T_\mathbb{C}^* M)$.
\end{proof}

Consequently, we consider the self-duality equations in the Sobolev setting as follows
\begin{equation}
(A, \Phi) \in N_1^2 \quad \text{such that} \quad
\begin{cases}
F_A^{\perp} + \big[\Phi, \Phi^*\big] = 0 \\
\mathrm{d}_A^{\prime\prime} \Phi = 0
\end{cases} \label{hitchin-sd}\text{a.e.}
\end{equation}
A pair $(A, \Phi) \in N_1^2$ that satisfies \eqref{hitchin-sd} is called a {\it solution of class $L_1^2$} to the self-duality equations.
\subsection{Gauge Groups}

\subsubsection{Smooth Case}
\begin{definition}
Define the group $\mathscr{G}$ of \textit{smooth unitary gauge transformations} on $(V, H)$ as
\[
\mathscr{G} := \Big\{ \varphi \in \Omega^0(M, \operatorname{End} V) \,\Big|\, \det \varphi = 1 \text{ and } \varphi \cdot \varphi^{\dagger_H} = \mathrm{id}_V \Big\},
\]
and the group ${\mathscr{G}}^\mathbb{C}$ of \textit{smooth complex gauge transformations} on $V$ as
\[
{\mathscr{G}}^\mathbb{C} := \Big\{ \varphi \in \Omega^0(M, \operatorname{End} V) \,\Big|\, \det \varphi = 1 \Big\}.
\]
\end{definition}

$\mathscr{G}$ is an infinite-dimensional real Lie group with Lie algebra 
\[
\mathfrak{Lie}(\mathscr{G}) = \Omega^{0, \text{sk}}(M, \operatorname{End}_0 V):=\Big\{ s \in \Omega^0(M, \operatorname{End}_0 V) \,\Big|\, s+s^{\dagger_H}=0 \Big\}.
\]
There is an action of $\mathscr{G}$ on $N$ given by  
\[
N \times \mathscr{G} \to N ; \quad (A, \Phi) \cdot g := g^{-1} (A, \Phi) g, \quad \text{for } g \in \mathscr{G} \text{ and } (A,\Phi) \in N.
\]
${\mathscr{G}}^\mathbb{C}$ is the complexification of $\mathscr{G}$ with $\mathfrak{Lie}({\mathscr{G}}^\mathbb{C}) = \Omega^0(M, \operatorname{End}_0 V)$. There is an action of ${\mathscr{G}}^{\mathbb{C}}$ on $N$ given by  
\begin{equation}
N \times \mathscr{G}^\mathbb{C} \to N ; \quad (A, \Phi) \cdot g := (\tilde{A}, g^{-1} \Phi g), \quad \text{for } g \in {\mathscr{G}}^{\mathbb{C}} \text{ and } (A,\Phi) \in N, \label{cg action}
\end{equation}
where  $\tilde{A}$ denotes the Chern connection on $(V, H)$ whose $(0,1)$-component is given by $g^{-1} \mathrm{d}_A^{\prime\prime} g$.
This ${\mathscr{G}}^\mathbb{C}$-action  extends the one of $\mathscr{G}$ on $N$. 

\subsubsection{Sobolev case}
Consider that $\mathscr{G}$ and $\mathscr{G}^\mathbb{C}$ are embedded in $\Omega^0(M, \operatorname{End} V)$. The group of \textit{unitary gauge transformations of class $L_2^2$}, denoted by $\mathscr{G}_2^2$, and the group of \textit{complex gauge transformations of class $L_2^2$}, denoted by ${\mathscr{G}^\mathbb{C}}_2^2$, are defined as the closures of $\mathscr{G}$ and $\mathscr{G}^\mathbb{C}$ in $L_2^2(M, \operatorname{End} V)$. We observe that if $(A, \Phi)\in N_1^2$ is a solution to the self-duality equations \eqref{hitchin-sd}, then so is $(A, \Phi) \cdot g$ for each $g \in \mathscr{G}_2^2$.

By the Sobolev embedding theorem, $L_2^2 \hookrightarrow C^{0, \gamma}$ for any $\gamma \in (0,1)$, so elements of $\mathscr{G}_2^2$ or ${\mathscr{G}^\mathbb{C}}_2^2$ can be considered pointwise by continuity. The Lie algebras are given by
\[
\mathfrak{Lie}(\mathscr{G}_2^2) = (L_2^2)^{\text{sk}}(M, \operatorname{End}_0 V) = \Big\{ \varphi \in L_2^2(M, \operatorname{End}_0 V) \ \Big|\ \varphi_z + \varphi_z^{\dagger_H} = 0 \quad\forall z\in M \Big\},
\]
and $\mathfrak{Lie}({\mathscr{G}^\mathbb{C}}_2^2) = L_2^2(M, \operatorname{End}_0 V)$. By the Sobolev multiplication theorem $L_2^2 \times L_1^2 \hookrightarrow L_1^2$, there are actions of $\mathscr{G}_2^2$ and ${\mathscr{G}^\mathbb{C}}_2^2$ on $N_1^2$, which are extensions of the actions of $\mathscr{G}$ and ${\mathscr{G}^\mathbb{C}}$ on $N$, respectively. 

\begin{lemma}\label{FF}
Take $(A,\Phi) \in N_1^2$ and $s \in \mathfrak{Lie}({\mathcal{G}^\mathbb{C}}_2^2)$. Locally, on a chart $(U, z)$, represent $\mathrm{d}_A^{\prime\prime} = \bar{\partial} + \theta$, then $\mathrm{d}_A^{\prime\prime}s$ is locally given by $\left(\frac{\partial s}{\partial \bar{z}} + [\theta, s]\right) \mathrm{d}\bar{z}.$ Denote by $(\mathrm{d}_A^{\prime\prime}s)^*$ the $L_1^2$ section of $\operatorname{End}_0V \otimes K_M$, which is locally given by $\left(\frac{\partial s}{\partial \bar{z}} + [\theta, s]\right)^{\dagger_H} \mathrm{d}z.$ 

The fundamental vector field induced by $s \in \mathfrak{Lie}({\mathcal{G}^\mathbb{C}}_2^2)$ with respect to the ${\mathcal{G}^\mathbb{C}}^2_2$ action on $N_1^2$ is given by
\[
V^s\Big|_{(A,\Phi)} = \Big(- (\mathrm{d}_A^{\prime\prime}s)^* + \mathrm{d}_A^{\prime\prime}s, [\Phi, s]\Big).
\]
\end{lemma}

\begin{proof}
Consider a curve in $N_1^2$ given by $\gamma(t) = (A_t, \Phi_t) := (A, \Phi) \cdot \exp(ts)$, and by definition, we have $V^s\big|_{(A,\Phi)} = \gamma^\prime(0)$. By the definition of the ${\mathcal{G}^\mathbb{C}}_2^2$ action on $N_1^2$, we have 
\[
\mathrm{d}_{A_t}^{\prime\prime} = \exp(-ts) \mathrm{d}_{A}^{\prime\prime} \exp(ts), \quad \text{and} \quad \Phi_t = \exp(-ts) \Phi \exp(ts).
\]
Locally, on the chart $(U, z)$, we write $\mathrm{d}_{A_t}^{\prime\prime} = \bar{\partial} + \exp(-ts) \bar{\partial} \big(\exp(ts)\big) + \exp(-ts) \theta \exp(ts) \mathrm{d}\bar{z}$ and $\Phi_t = \exp(-ts) \phi \exp(ts) \mathrm{d}z$. Taking the limit as $t \to 0$, we compute
\[
\lim_{t \to 0} \frac{\mathrm{d}_{A_t}^{\prime\prime} - \mathrm{d}_A^{\prime\prime}}{t} = \left(\frac{\partial s}{\partial \bar{z}} + [\theta, s]\right) \mathrm{d}\bar{z}=\mathrm{d}_A^{\prime\prime}s\text{ locally},
\]
Similarly, we have $\lim\limits_{t \to 0} \frac{\Phi_t - \Phi}{t} = [\phi, s] \mathrm{d}z,$ which coincides exactly with the form $[\Phi, s]$.

By the skew-Hermitian property of $T_A \mathscr{A}_1^2$, the term  $\lim\limits_{t \to 0} \frac{\mathrm{d}_{A_t} - \mathrm{d}_A}{t}$ is given by 
\[
-\left(\lim\limits_{t \to 0} \frac{ \mathrm{d}_{A_t}^{\prime\prime} - \mathrm{d}_A^{\prime\prime}}{t}\right)^*+\lim\limits_{t \to 0} \frac{\mathrm{d}_{A_t}^{\prime\prime} - \mathrm{d}_A^{\prime\prime}}{t} .
\] 
Finally, we conclude $\gamma^\prime(0) = \Big(- (\mathrm{d}_A^{\prime\prime}s)^* + \mathrm{d}_A^{\prime\prime}s, [\Phi, s]\Big).$
\end{proof}

\subsubsection{Higgs stable pair}
\begin{definition}
A pair $(A, \Phi) \in N$ is called \textit{Higgs (semi-)stable} if $\mathrm{d}_A^{\prime\prime} \Phi = 0$ and any holomorphic line subbundle $L$ of $(V, \mathrm{d}_A^{\prime\prime})$ satisfying $\Phi(L) \subseteq L \otimes K_M$ must satisfy:
\[
\deg L < \left(\leq\right) \frac{\deg V}{2}.
\]
\end{definition}

\begin{lemma}
If $(A, \Phi)$ lies in a ${\mathscr{G}^\mathbb{C}}_2^2$-orbit of a Higgs semi-stable pair $(B, \Psi) \in N$, then $\mathrm{d}_A^{\prime\prime} \Phi = 0$.
\end{lemma}

\begin{proof}
By the definition of the ${\mathscr{G}^\mathbb{C}}_2^2$ action on $N_1^2$ given in \eqref{cg action}, we have:
\[
\mathrm{d}_A^{\prime\prime} = g^{-1} \mathrm{d}_B^{\prime\prime} g, \quad \Phi = g^{-1} \Psi g,
\]
for some $g \in {\mathscr{G}^\mathbb{C}}_2^2$. Consequently $\mathrm{d}_A^{\prime\prime} \Phi = g^{-1} (\mathrm{d}_B^{\prime\prime} \Psi) g = 0.$
\end{proof} 

\begin{definition}\label{Hstable}
A pair $(A, \Phi) \in N_1^2$ is called \textit{Higgs (semi-)stable} if it lies in a ${\mathscr{G}^\mathbb{C}}_2^2$-orbit of some Higgs (semi-)stable pair in $N$.
\end{definition}

\subsection{Moment Map}

The moment map $\mu: N \to \mathfrak{Lie}(\mathscr{G})^*$ of the action of the Lie group $\mathscr{G}$ on $N$ can be explicitly computed, with detailed calculations provided in~\cite[Section~9]{atiyah1983yang} and~\cite[Section~4]{hitchin1987self}. This result can naturally extend to the case of $\mathscr{G}_2^2$ action on $N_1^2$ as
\begin{equation}
\mu: N_1^2 \to \mathfrak{Lie}(\mathscr{G}_2^2)^*, \quad (A, \Phi) \mapsto \mu(A, \Phi)(s) := \int_M \operatorname{Tr}\Big( \big(F_A^\perp + [\Phi, \Phi^*]\big) s \Big). \label{moment map}
\end{equation}
For convenience, we write the moment map as $\mu(A, \Phi) = F_A^\perp + [\Phi, \Phi^*]$. 

Introduce the standard  inner product on $(L^2)^{\text{sk}}(M, \operatorname{End}_0 V)$, which induces an inner product on the linear subspace $\mathfrak{Lie}(\mathscr{G}_2^2) = (L_2^2)^{\text{sk}}(M, \operatorname{End}_0 V)$ as
\[
\langle s, t \rangle_{L^2} = \int_M \operatorname{Tr}\big(s^{\dagger_H} t + t^{\dagger_H} s\big) \, \omega_M, \quad s, t \in (L^2)^{\text{sk}}(M, \operatorname{End}_0 V).
\]
Define the map $\mu^\sharp : N_1^2 \to (L^2)^{\text{sk}}(M, \operatorname{End}_0 V)$ as
\[
\mu^\sharp(A, \Phi) = \frac{F_A^\perp + [\Phi, \Phi^*]}{\omega_M}.
\]
Whenever $x \in N_1^2$ is a pair such that $\mu(x)$ has $L_2^2$ regularity, then $\mu^\sharp(x)$ can be regarded as an element in  $\mathfrak{Lie}(\mathscr{G}_2^2) = (L_2^2)^{\text{sk}}(M, \operatorname{End}_0 V)$. 
Finally, we define the energy function as
\[
\|\mu\|^2_{L^2} : N_1^2 \to \mathbb{R}_{\geq 0}, \quad (A, \Phi) \mapsto \Big\langle \mu^\sharp(A, \Phi), \mu^\sharp(A, \Phi) \Big\rangle_{L^2}.
\]

\subsection{Atiyah-Bott's Result}

One of the key tools employed in this note is a result from Atiyah-Bott~\cite[Lemma 14.8-9]{atiyah1983yang}, which provides a method for finding smooth connections within the ${\mathscr{G}^\mathbb{C}}_2^2$-orbit in $\mathscr{A}_1^2$.

\begin{lemma}[Atiyah-Bott] \label{A-B}
For any connection $A \in \mathscr{A}_1^2$, the set consisting of smooth connections in the ${\mathscr{G}^\mathbb{C}}_2^2$-orbit of $A$ is a dense subset of this orbit. Furthermore, if $A_1, A_2$ are smooth and $A_1 = A_2 \cdot g$ for some $g \in {\mathscr{G}^\mathbb{C}}_2^2$, then $g \in \mathscr{G}^\mathbb{C}$.
\end{lemma}

For any pair $(A, \Phi) \in N_1^2$, based on the lemma above, we can find smooth pairs in $N$ within any stable ${\mathscr{G}^\mathbb{C}}_2^2$-orbits in $N_1^2$ as follows:

\begin{lemma}\label{AB}
Let $(A, \Phi) \in N_1^2$ satisfy $\operatorname{d}_A^{\prime\prime} \Phi = 0$. Denote  by $\mathcal{O}_{(A, \Phi)}$ the ${\mathscr{G}^\mathbb{C}}_2^2$-orbit of $(A, \Phi)$ in $N_1^2$. Then the subset of $\mathcal{O}_{(A, \Phi)}$ consisting of smooth pairs is dense in $\mathcal{O}_{(A, \Phi)}$. Moreover, if both $(A_1, \Phi_1)$ and $(A_2, \Phi_2)$ are contained in $N$ and $(A_1, \Phi_1) = (A_2, \Phi_2) \cdot g$ for some $g \in {\mathscr{G}^\mathbb{C}}_2^2$, then $g \in \mathscr{G}^\mathbb{C}$.
\end{lemma}

\begin{proof}
By Lemma~\ref{A-B}, there exists $g \in {\mathscr{G}^\mathbb{C}}_2^2$ such that $\tilde{A} := A \cdot g$ is smooth, and thus $(\tilde{A}, \tilde{\Phi}) := (A, \Phi) \cdot g$ is a smooth pair in $\mathcal{O}_{(A, \Phi)}$. The smoothness of $\tilde{\Phi}$ follows from the elliptic equation $\mathrm{d}_{\tilde{A}}^{\prime\prime} \tilde{\Phi} = 0$. Therefore, the set $\Big\{ (\tilde{A}, \tilde{\Phi}) \cdot g \mid g \in \mathscr{G}^\mathbb{C} \Big\}$ is a dense subset of the ${\mathscr{G}^\mathbb{C}}_2^2$-orbit $\mathcal{O}_{(A, \Phi)}$.
\end{proof}
\subsection{Main Result}

Hitchin’s result~\cite[Theorem 4.3]{hitchin1987self} includes both the existence and uniqueness of solutions to the self-duality equations. Here, we focus only on the existence, which is stated as follows.

\begin{theorem}[Hitchin]\label{main}
For any Higgs stable pair $(A, \Phi) \in N$, there exist smooth solutions to the Hitchin self-duality equations~\eqref{hitchin-sd} within the ${\mathscr{G}^\mathbb{C}}_2^2$-orbit of $(A, \Phi)$.
\end{theorem}

We divide the proof of this theorem into two main steps. First, we seek a solution to the self-duality equations in the orbit $\mathcal{O}_{(A,\Phi)}$, where the solution is not necessarily smooth. Second, we show the existence of a smooth pair in the $\mathscr{G}_2^2$-orbit of this solution. These two steps correspond to the following theorems:

\begin{theorem}\label{main1}
For any Higgs stable pair $(A, \Phi) \in N$, there exist solutions of class $L_1^2$ to the Hitchin self-duality equations~\eqref{hitchin-sd} within the ${\mathscr{G}^\mathbb{C}}_2^2$-orbit of $(A, \Phi)$. 
\end{theorem}

\begin{theorem}\label{main2}
If $(A, \Phi) \in N_1^2$ is a solution of the self-duality equations~\eqref{hitchin-sd}, then there exists a smooth pair in the $\mathscr{G}_2^2$-orbit of $(A, \Phi)$.
\end{theorem}

\section{Reduction of Solving Self-Duality Equations}

Let $(A, \Phi) \in N$ be a smooth Higgs stable pair, and denote by $\mathcal{O}_{(A, \Phi)}$  the ${\mathscr{G}^\mathbb{C}}_2^2$-orbit of $(A, \Phi)$ in $N_1^2$.  The main goal of this section is to establish four mutually equivalent conditions for the existence of a solution to the self-duality equations within the orbit $\mathcal{O}_{(A, \Phi)}$. Hitchin's original work introduced these equivalent conditions under the assumption of smoothness, transforming the problem of finding solutions to the self-duality equations into a more tractable form. In this section, we remove the smoothness assumption by extending the discussion to a general Higgs stable pair in $\mathcal{O}_{(A, \Phi)}$. These conditions are summarized in the following proposition.

\begin{prop}\label{eq-conditions}
For $x \in \mathcal{O}_{(A, \Phi)}$, the following conditions are mutually equivalent:
\begin{enumerate}
    \item $x$ is a solution of class $L_1^2$ to the Hitchin self-duality equation~\eqref{hitchin-sd}.
    \item $x$ is a zero of the moment map $\mu$ as defined in~\eqref{moment map}.
    \item $x$ is a critical point of $\|\mu\|_{L^2}^2: N_1^2 \to \mathbb{R}_{\geq0}$.
    \item $x$ is a critical point of the restricted function $\left.\|\mu\|_{L^2}^2\right|_{\mathcal{O}_{(A, \Phi)}}: \mathcal{O}_{(A, \Phi)} \to \mathbb{R}_{\geq0}$.
\end{enumerate}
\end{prop}

Since $(N_1^2, g)$ is a Banach manifold, we need to redefine the notion of critical point. First, we introduce a general result on Sobolev sections of vector bundles from~\cite[Section 4.1]{wells2008differential} and~\cite[Chapter~4,~Proposition 3.2]{taylorpartial} as follows

\begin{lemma}\label{dualsobolev}
Let $E$ be a vector bundle over a compact manifold $M$ and $g$ a bundle metric. For any bounded linear functional $f$ on $L_k^2(M, E)$, $k \in \mathbb{R}$, there exists a unique $t \in L_{-k}^2(M, E)$ such that: $g(s, t) = f(s),  \,\forall s \in L_k^2(M, E).$ Moreover, the inequality $|g(s, t)| \leq \|s\|_{L_k^2} \|t\|_{L_{-k}^2}$ holds.
\end{lemma}

Regard the tangent space at $x \in N_1^2$ as the space of $L_1^2$ sections of a vector bundle:  
\[
T_x N_1^2 \cong (L_1^2)^{1, \mathrm{sk}}(M, \mathrm{End}_0 V) \oplus (L_1^2)^{1,0}(M, \mathrm{End}_0 V).
\]
On the other hand, any $\theta = \theta^{1,0} + \theta^{0,1}$ in $(L_1^2)^{1, \mathrm{sk}}(M, \mathrm{End}_0 V)$ satisfies $\theta^{1,0} = -\left(\theta^{0,1}\right)^*$. Thus, $\theta$ is determined by its $(0,1)$ component. Therefore, we can regard $T_x N_1^2$ as the vector space $L_1^2(M, \mathrm{End}_0 V \otimes T^*_{\mathbb{C}} M)$.  
The Kähler metric $g$ on $N_1^2$ can then be viewed as a bundle metric on $\mathrm{End}_0 V \otimes T^*_{\mathbb{C}} M$.

\begin{lemma}
Denote by $\mathrm{d}_x f$ the differential 1-form in $(T_x N_1^2)^*$ associated with the smooth function $f$. Then $\mathrm{d}_x \|\mu\|^2_{L^2}$ is a bounded linear functional on the Hilbert space $(T_x N_1^2, \|\cdot\|_{L_1^2})$.
\end{lemma}

\begin{proof}
Let $x := (A, \Phi) \in N_1^2$ and $\mu(x) = F_A^\perp + [\Phi, \Phi^*]$. Let $Y = (\dot{A}, \dot{\Phi})$ be a tangent vector in $T_x N_1^2$. Then $\mathrm{d}_x \mu(Y)$ is an $L^2$ section of $\mathrm{End}_0 V \otimes \Lambda^2 T^*_{\mathbb{C}} M$, given by:
\begin{align}
\mathrm{d}_x \mu(Y) &= \lim_{t \to 0} \frac{1}{t} \Big( \mu(x + tY) - \mu(x) \Big) \nonumber \\
&= \lim_{t \to 0} \frac{1}{t} \Big( \big(F_{A + t\dot{A}} - F_A\big) + \big([\Phi + t\dot{\Phi}, \Phi^* + t\dot{\Phi}^*] - [\Phi, \Phi^*]\big) \Big) \nonumber \\
&= d_A \dot{A} + [\Phi, \dot{\Phi}^*] + [\dot{\Phi}, \Phi^*]. \label{dmu}
\end{align}

Suppose $\{Y_i\} := (A_i, \Phi_i)$ converges to $Y$ in $L_1^2$. Using the embedding $L_1^2 \times L_1^2 \hookrightarrow L^2$, we have $\{\mathrm{d}_x \mu(Y_i)\}$ converging to $\mathrm{d}_x \mu(Y)$ in $L^2$. On the other hand,
\[
\mathrm{d}_x \|\mu\|^2_{L^2}(Y) = 2 \left\langle \mathrm{d}_x \mu^\sharp(Y), \mu^\sharp(x) \right\rangle_{L^2}
\]
by definition. Hence, the sequence $\mathrm{d}_x \|\mu\|^2_{L^2}(Y_i)$ converges to $\mathrm{d}_x \|\mu\|^2_{L^2}(Y)$.
\end{proof}

\begin{definition}
$\operatorname{grad}_x \|\mu\|_{L^2}^2$ is defined as the $L_{-1}^2$ section of $\mathrm{End}_0 V \otimes  T^*_{\mathbb{C}} M$ such that
\[
g\left(\operatorname{grad}_x \|\mu\|^2_{L^2}, Y\right) = \mathrm{d}_x\|\mu\|^2_{L^2}(Y), \quad \forall \, Y \in T_x N_1^2.
\]
We call $x$ a \textit{critical point} of $\|\mu\|_{L^2}^2$ if
$
\operatorname{grad}_x \|\mu\|_{L^2}^2 = 0.
$
\end{definition}

\subsection{Proof of Proposition~\ref{eq-conditions}}

\subsubsection{$(1) \Leftrightarrow (2)$}
The equivalence between (1) and (2) is a trivial observation since $(A, \Phi)$ is stable, and for any $(\tilde{A}, \tilde{\Phi}) \in \mathcal{O}_{(A, \Phi)}$, the second part of the self-duality equations, $\mathrm{d}_{\tilde{A}}^{\prime\prime} \tilde{\Phi} = 0$, automatically holds.

\subsubsection{$(2) \Leftrightarrow (3)$}
For $x \in N_1^2$, $\operatorname{grad}_x \|\mu\|_{L^2}^2$ cannot, in general, be regarded as a tangent vector in $T_x N_1^2$. However, under certain regularity condition, we have the following

\begin{lemma}\label{grad}
Consider the Banach K\" ahler manifold $(N_1^2, \omega, g, \mathcal{I})$ with $\mathscr{G}_2^2$ acting on it with the moment map $\mu$ defined in~\eqref{moment map}. Suppose $x \in N_1^2$ such that $\mu^\sharp(x) \in (L_2^2)^{\rm sk}(M, \operatorname{End}_0 V)$. Then $\operatorname{grad}_x \|\mu\|^2_{L^2}$ is the tangent vector in $T_x N_1^2$ given by
\[
\operatorname{grad}_x \|\mu\|^2_{L^2} = \left. 2\mathcal{I} V^{\mu^\sharp(x)} \right|_x = 2\mathrm{i} \Big( \big( \operatorname{d}_A^{\prime\prime} \mu^\sharp(x) \big)^* + \operatorname{d}_A^{\prime\prime} \mu^\sharp(x), \big[\Phi, \mu^\sharp(x)\big] \Big),
\]
where $V^{\mu^\sharp(x)}$ is the fundamental vector field on $N_1^2$ induced by $\mu^\sharp(x) \in \mathfrak{Lie}(\mathscr{G}_2^2)$.
\end{lemma}

\begin{proof}
For each $s \in \mathfrak{Lie}(\mathscr{G}_2^2)$, define $\mu_s: N_1^2 \to \mathbb{R}$ by $x \mapsto \mu(x)(s)$. Then, for any $Y \in T_x N_1^2$
\[
\mathrm{d}_x\mu(Y)(s) = \mathrm{d}_x\mu_s(Y) = \omega(V^s\big|_x, Y) = g(\mathcal{I}V^s\big|_x, Y), \quad
\]

On the other hand,
\[
g\left(\operatorname{grad}_x \|\mu\|^2_{L^2}, Y\right) = \mathrm{d}_x\|\mu\|^2_{L^2}(Y) = 2 \left\langle \mathrm{d}_x\mu^\sharp(Y), \mu^\sharp(x) \right\rangle_{L^2} = 2 \mathrm{d}_x\mu(Y)\left(\mu^\sharp(x)\right)
\]

For any $x = (A, \Phi) \in N_1^2$, we can explicitly express $V^s\big|_x = \Big(- (\mathrm{d}_A^{\prime\prime}s)^* + \mathrm{d}_A^{\prime\prime}s, [\Phi, s]\Big)$ by Lemma \ref{FF}. Since $\mu^\sharp(x) \in (L_2^2)^{\text{sk}}(M, \operatorname{End}_0 V) = \mathfrak{Lie}(\mathscr{G}_2^2)$, combining the above two expressions, it follows that
\[
\operatorname{grad}_x \|\mu\|^2_{L^2} = \left.2\mathcal{I} V^{\mu^\sharp(x)}\right|_x = 2\mathrm{i}\Big(\big(\mathrm{d}_A^{\prime\prime} \mu^\sharp(x)\big)^* + \mathrm{d}_A^{\prime\prime} \mu^\sharp(x), \big[\Phi, \mu^\sharp(x)\big]\Big).
\]
\end{proof}

The implication (2) $\Rightarrow$ (3) is straightforward by the lemma above. Next, we focus on how (3) implies (2). On the basis of the previous lemma, we can now analyze $\operatorname{grad}_x \|\mu\|^2_{L^2}$ in the general case for $x \in N_1^2$ as follows:

\begin{lemma}
Consider the Banach K\"ahler manifold $(N_1^2, \omega, g, \mathcal{I})$ with $\mathscr{G}_2^2$ acting on it with the moment map $\mu$ defined in~\eqref{moment map}. For any $x \in N_1^2$, the $\operatorname{grad}_x \|\mu\|^2_{L^2}$ is the $L_{-1}^2$ section of the vector bundle $\mathrm{End}_0 V \otimes T_{\mathbb{C}}^* M$ formally given by
\[
\operatorname{grad}_x \|\mu\|^2_{L^2} = 2 \mathrm{i} \Big( \big( \operatorname{d}_A^{\prime\prime} \mu^\sharp(x) \big)^* + \operatorname{d}_A^{\prime\prime} \mu^\sharp(x), \big[\Phi, \mu^\sharp(x)\big] \Big).
\]
\end{lemma}

\begin{proof}
For a smooth pair $x \in N$, the result follows directly from the previous lemma. For a general $x = (A, \Phi) \in N_1^2$, let $\{x_i = (A_i, \Phi_i)\} \subset N_1^2$ be a sequence of smooth pairs converging to $x$ in $L_1^2$. By \cite[Lemma~1.1]{uhlenbeck1982connections} and the embedding $L_1^2 \times L_1^2 \hookrightarrow L^2$, both $F_{A_i}$ converge to $F_A$ and $[\Phi_i, \Phi_i^*]$ converge to $[\Phi, \Phi^*]$ in $L^2$. Therefore,
\[
\mu(x_i) = F_{A_i} + [\Phi_i, \Phi_i^*] \xrightarrow{L^2} F_A + [\Phi, \Phi^*] = \mu(x).
\]
For any fixed $Y = (\dot{A}, \dot{\Phi}) \in T_x N_1^2$, we have 
\[
\mathrm{d}_x \mu(Y) = \mathrm{d}_A \dot{A} + [\dot{\Phi}, \Phi^*] + [\Phi, \dot{\Phi}^*],
\]
by \eqref{dmu}. Consequently, $\mathrm{d}_{x_i} \mu(Y) \to \mathrm{d}_x \mu(Y)$ in $L^2$. Then
\begin{align*}
g\big(\operatorname{grad}_{x_i}\|\mu\|^2_{L^2}, Y\big) 
&= \mathrm{d}_{x_i} \|\mu\|^2_{L^2}(Y) = 2 \left\langle \mathrm{d}_{x_i} \mu^\sharp(Y), \mu^\sharp(x_i) \right\rangle_{L^2} \\
&\to 2\left\langle \mathrm{d}_x \mu^\sharp(Y), \mu^\sharp(x) \right\rangle_{L^2} = \mathrm{d}_x \|\mu\|^2_{L^2}(Y) = g\big(\operatorname{grad}_x\|\mu\|^2_{L^2}, Y\big).
\end{align*}
By Lemma~\ref{dualsobolev}, any bounded functional on $L_{-1}^2$ can be written in the form $g(\cdot, Y)$ for some $Y$. Therefore
\[
\operatorname{grad}_{x_i}\|\mu\|^2_{L^2} \rightharpoonup \operatorname{grad}_x\|\mu\|^2_{L^2} \quad \text{in } L_{-1}^2.
\]
On the other hand, by the embedding $L_1^2 \times L_1^2 \hookrightarrow L^2$, we have $[\Phi_i, \mu^\sharp(x_i)] \to [\Phi, \mu^\sharp(x)]$ in $L^2$, and $\operatorname{d}_{A_i}^{\prime\prime} \mu^\sharp(x_i) \to \operatorname{d}_A^{\prime\prime} \mu^\sharp(x)$ in $L_{-1}^2$. Then
\begin{align*}
\operatorname{grad}_{x_i}\|\mu\|^2_{L^2}&=2\mathrm{i} \Big( \big( \operatorname{d}_{A_i}^{\prime\prime} \mu^\sharp(x_i) \big)^* + \operatorname{d}_{A_i}^{\prime\prime} \mu^\sharp(x_i), \big[\Phi_i, \mu^\sharp(x_i)\big] \Big) \\
&\xrightarrow{L_{-1}^2} 
2\mathrm{i} \Big( \big( \operatorname{d}_A^{\prime\prime} \mu^\sharp(x) \big)^* + \operatorname{d}_A^{\prime\prime} \mu^\sharp(x), \big[\Phi, \mu^\sharp(x)\big] \Big).
\end{align*}
By the uniqueness of weak limits, we conclude
\[
\operatorname{grad}_x \|\mu\|^2_{L^2} = 2\mathrm{i} \Big( \big( \operatorname{d}_A^{\prime\prime} \mu^\sharp(x) \big)^* + \operatorname{d}_A^{\prime\prime} \mu^\sharp(x), \big[\Phi, \mu^\sharp(x)\big] \Big).
\]
This completes the proof.
\end{proof}

\begin{cor}\label{cp-reg}
If $x \in N_1^2$ is a critical point of $\|\mu\|^2_{L^2}$, then $\mu^\sharp(x)$ lies in $ (L_2^2)^{\mathrm{sk}}(M, \operatorname{End}_0 V)=\mathfrak{Lie}(\mathscr{G}_2^2)$, and $V^{\mu^\sharp(x)} = 0$.
\end{cor}

\begin{proof}
$\operatorname{d}_A^{\prime\prime} \mu^\sharp(x) = 0$ since $x$ is a critical point of $\|\mu\|^2_{L^2}$. Locally, on a chart $(U, z)$, we write
\[
A = \mathrm{d} - \theta^{\dagger_H}\text{d}z + \theta \text{d}\bar z, \quad \mu^\sharp(x) = m,
\]
Using the local equation $\bar{\partial} m = -[\theta, m]\text{d}\bar z$ and Corollary~\ref{SobM}, which states that $L_1^2 \times L^2 \hookrightarrow L^{2-\varepsilon}$ for any $0 < \varepsilon < \frac{1}{2}$, we deduce that both $m$ and $\bar{\partial} m$ have $L^{2-\varepsilon}$ regularity. Then, by elliptic estimates of the operator $\bar\partial$, $m$ has $L_1^{2-\varepsilon}$ regularity. By iteratively applying the Corollary~\ref{SobM}, $L_1^2 \times L_1^{2-\varepsilon} \hookrightarrow L_1^{2-2\varepsilon}$ and elliptic regularity, we conclude that $m$ has $L_2^{2-2\varepsilon}$ regularity. Using the Sobolev multiplication theorem mentioned in Lemma \ref{Sob property}, we have
\[
L_2^{2-2\varepsilon} \times L_1^2 \hookrightarrow L_1^2,
\]
which implies that $\bar{\partial} m$ has $L_1^2$ regularity. Using the elliptic estimation, $m \in L_2^2$, and consequently, $\mu^\sharp(x)$ has $L_2^2$ regularity. By Lemma~\ref{grad}, we have
\[
\operatorname{grad}_x \|\mu\|^2_{L^2} = \left. 2 \mathcal{I} V^{\mu^\sharp(x)} \right|_x = 0,
\]
which implies $V^{\mu^\sharp(x)} = 0$.
\end{proof}

Corollary \ref{cp-reg} utilizes elliptic regularity techniques, which are commonly employed throughout this note. Next, we will prove that if $x$ is a Higgs stable pair, then $V^{\mu^\sharp(x)} = 0$ if and only if $\mu^\sharp(x) = 0$, thereby completing the proof of $(2) \Leftrightarrow (3)$. To achieve this, we introduce the following result from Hitchin~\cite[Theorem 3.15]{hitchin1987self}. 

\begin{theorem}[Hitchin]\label{Hitchin-315}
Let $(A_1, \Phi_1)$ and $(A_2, \Phi_2)$ be Higgs semistable smooth pairs in $N$, with at least one of them being stable. Let
\[
\psi: (V, \operatorname{d}^{\prime\prime}_{A_1}) \to (V, \operatorname{d}^{\prime\prime}_{A_2})
\]
be a non-zero holomorphic bundle homomorphism such that $\psi \Phi_1 = \Phi_2 \psi$. Then $\psi$ is an isomorphism. Moreover, if $(A_1, \Phi_1) = (A_2, \Phi_2)$, then $\psi$ is scalar multiplication.
\end{theorem}

The following lemma provides a slight generalization of Hitchin’s result, extending its applicability to the non-smooth setting.

\begin{lemma}\label{Hitchin315}
Let $(A_1, \Phi_1)$ and $(A_2, \Phi_2)$ be Higgs semistable pairs in $N_1^2$, with at least one of them being stable. Let $\psi: V \to V$ be a non-zero bundle homomorphism such that 
\[
\psi \operatorname{d}^{\prime\prime}_{A_1} = \operatorname{d}^{\prime\prime}_{A_2} \psi, \quad \text{and} \quad \psi \Phi_1 = \Phi_2 \psi.
\]
Then, $\psi$ is an isomorphism with $L_2^2$ regularity. Moreover, if $(A_1, \Phi_1) = (A_2, \Phi_2)$, then $\psi$ is scalar multiplication.
\end{lemma}

\begin{proof}
By Lemma~\ref{AB}, there exist $g_1, g_2 \in {\mathscr{G}^\mathbb{C}}_2^2$ such that: $(\tilde{A}_1, \tilde{\Phi}_1) := (A_1, \Phi_1) \cdot g_1$ and $ (\tilde{A}_2, \tilde{\Phi}_2) := (A_2, \Phi_2) \cdot g_2$ are smooth pairs. If $\psi$ is a non-zero bundle homomorphism satisfying $\psi \operatorname{d}^{\prime\prime}_{A_1} = \operatorname{d}^{\prime\prime}_{A_2} \psi$ and $\psi \Phi_1 = \Phi_2 \psi$, then the induced map 
\[
\tilde{\psi} := g_2^{-1} \psi g_1 : (V, \operatorname{d}^{\prime\prime}_{\tilde{A}_1}) \to (V, \operatorname{d}^{\prime\prime}_{\tilde{A}_2})
\]
is a non-zero holomorphic bundle homomorphism satisfying $\tilde{\psi} \tilde{\Phi}_1 = \tilde{\Phi}_2 \tilde{\psi}$. By Hitchin's Theorem~\ref{Hitchin-315}, $\tilde{\psi}$ is an isomorphism, and therefore $\psi$ is also an isomorphism with $L_2^2$ regularity.
\end{proof}

\begin{cor}\label{s=0}
Let $(A, \Phi)$ be a Higgs stable pair in $N_1^2$, and let $x \in \mathcal{O}_{(A, \Phi)}$. Then, for any $s \in \mathfrak{Lie}({\mathscr{G}^\mathbb{C}}_2^2)$, the condition $V^s|_x = 0$ holds if and only if $s = 0$. 
\end{cor}

Combining Corollary~\ref{cp-reg} and Corollary~\ref{s=0}, we have completed the proof of $(3) \Rightarrow (2)$. Thus, we have established the equivalence between $(2)$ and $(3)$.

\subsubsection{$(3) \Leftrightarrow (4)$}

The implication $(3) \Rightarrow (4)$ is straightforward. Next, we consider the converse direction $(4) \Rightarrow (3)$. We have shown that if $x \in \mathcal{O}_{(A, \Phi)}$ is a smooth pair, then
\[
\operatorname{grad}_x\left(\left.\|\mu\|_{L^2}^2\right|_{\mathcal{O}_{(A, \Phi)}}\right) = \operatorname{grad}_x \|\mu\|^2_{L^2},
\]
since $\operatorname{grad}_x \|\mu\|^2_{L^2} = \left.2\mathcal{I}V^{\mu^\sharp(x)}\right|_x$ is a tangent vector of $\mathcal{O}_{(A, \Phi)}$. 

For a general $x \in \mathcal{O}_{(A, \Phi)}$, take a smooth sequence $\{x_i = x \cdot g_i\}_i$ convergent to $x$, where $g_i\in {\mathscr{G}^\mathbb{C}}_2^2$ converges to the identity. Consider $T_x \mathcal{O}_{(A, \Phi)}$ as a closed subset of $T_x N_1^2$, and define
\[
v := 2\mathrm{i} \Big( \big( \operatorname{d}_A^{\prime\prime} \mu^\sharp(x) \big)^* + \operatorname{d}_A^{\prime\prime} \mu^\sharp(x), \big[\Phi, \mu^\sharp(x)\big] \Big), \quad v_i := 2 \mathcal{I} V^{\mu^\sharp(x_i)}\Big|_{x_i}.
\]
We claim that $T_x \mathcal{O}_{(A, \Phi)}$ is a closed subspace of the Hilbert space $(T_x N_1^2, \|\cdot\|_{L_1^2})$ and will prove this later in Lemma \ref{closed}. By the Riesz representation theorem, we regard $(T_x \mathcal{O}_{(A, \Phi)})^*$ as a closed subspace of $(T_x N_1^2)^*$. What remains to show is that $g(v, \cdot)$ belongs to $(T_x \mathcal{O}_{(A, \Phi)})^*$.

On one hand, the vectors $v_i$ are tangent vectors in $T_{x_i} \mathcal{O}_{(A, \Phi)}$, implying that $(R_{g_i})_* v_i$ are tangent vectors in $T_x \mathcal{O}_{(A, \Phi)}$. Consequently, $g((R_{g_i})_* v_i, \cdot) \in (T_x \mathcal{O}_{(A, \Phi)})^*$. On the other hand, since $v_i$ converges to $v$ in $L_{-1}^2(M, \operatorname{End}_0 V \otimes T^*_{\mathbb{C}} M)$ and $g_i$ converges to the identity, it follows that $g((R_{g_i})_* v_i, \cdot)$ converges to $g(v, \cdot)$ in $(T_x N_1^2)^*$. By the closedness of $(T_x \mathcal{O}_{(A, \Phi)})^*$, we conclude that $g(v, \cdot) \in (T_x \mathcal{O}_{(A, \Phi)})^*.$ This completes the proof.

\begin{lemma}\label{closed}
Let $x = (A, \Phi) \in N_1^2$ be a Higgs stable pair, and let $\mathcal{O}_x$ be the ${\mathscr{G}^\mathbb{C}}_2^2$-orbit of $x$. Then $T_x \mathcal{O}_x$ is a closed subspace of the Hilbert space $(T_x N_1^2, \|\cdot\|_{L_1^2})$.
\end{lemma}

\begin{proof}
Suppose $\{s_i\}$ is a sequence in $\mathfrak{Lie}({\mathscr{G}^\mathbb{C}}_2^2)$ such that the fundamental vector $V^{s_i}|_x$ converges to $X \in T_x N_1^2$. Then by Lemma~\ref{FF}, $\{\mathrm{d}_A^{\prime\prime}s_i\}$ converge in $(L_1^2)^{0,1}(M, \operatorname{End}_0 V)$. We only need to show that $X \in T_x \mathcal{O}_x$, which is equivalent to finding $s \in \mathfrak{Lie}({\mathscr{G}^\mathbb{C}}_2^2)$ such that $V^s|_x=X$. First we claim that
\begin{claim}
\it $\{s_i\}$ is bounded in $L^2$
\end{claim}
If not, we may assume that $\|s_i\|_{L^2} \to \infty$. Take $s_i = \lambda_i \tilde{s}_i$ such that $\|\tilde{s}_i\|_{L^2} = 1$ and $\lambda_i = \|s_i\|_{L^2} \to \infty$. The fundamental vector of $\tilde s_i$ satisfies
\[
V^{\tilde{s}_i}\big|_x = \lambda_i^{-1} V^{s_i}\big|_x \to 0
\]
since $V^{s_i}$ converges to $X$ in $L_1^2$, which in turn gives
\begin{equation}
\mathrm{d}_A^{\prime\prime} \tilde{s}_i \xrightarrow{L_1^2} 0 \quad \text{and} \quad [\Phi, \tilde{s}_i] \xrightarrow{L_1^2} 0. \label{*3}
\end{equation}
Hence, $\mathrm{d}_A^{\prime\prime} \tilde{s}_i$ is bounded in $L_1^2$. By the elliptic estimates for $\mathrm{d}_A^{\prime\prime}$, $\tilde{s}_i$ is bounded in $L_2^2$. Therefore, we may assume $\{\tilde{s}_i\}$ has a weak limit $t$ in $L_2^2(M, \operatorname{End}_0 V)$. On the one hand, since $L_2^2$ compactly embeds into $L^2$, $\tilde{s}_i$ converges to $t$ in $L^2$. Together with $\|\tilde{s}_i\|_{L^2} = 1$, this implies $t \neq 0$. On the other hand, in the limit of \eqref{*3}
\[
\mathrm{d}_A^{\prime\prime} t = 0 \quad \text{and} \quad [\Phi, t] = 0.
\]
By Lemma~\ref{FF}, this gives $V^t|_x = 0$. Furthermore, by Corollary~\ref{s=0}, we conclude that $t = 0$, which is a contradiction. Therefore, $\|s_i\|_{L^2}$ must be bounded, and the proof of the claim is complete.

$\left\|\mathrm{d}_A^{\prime\prime}s_i\right\|_{L_1^2}$ is bounded since the sequence $\mathrm{d}_A^{\prime\prime}s_i$ converges in $L_1^2$. $\|s_i\|_{L^2}$ is bounded by the claim above. Then, by elliptic estimate for the operator $\mathrm{d}_A^{\prime\prime}$, $\|s_i\|_{L_2^2}$ is bounded. Hence, we may assume that the sequence $\{s_i\}$ weakly converges to a $s$ in $L_2^2(M, \operatorname{End}_0 V)$. Thus, $V^{s_i}|_x$ weakly converges to $V^s|_x$ in $L_1^2(M, \operatorname{End}_0 V \otimes T_\mathbb{C}^* M)$. Note that $V^{s_i}|_x$ converges to $X \in T_x N_1^2$, by the uniqueness of weak limits, we have $V^s = X$. This completes the proof of the lemma.
\end{proof}

\section{Solutions in Higgs Stable Orbit}

In this section, we aim to prove Theorem~\ref{main1}. In the previous section, we completed the proof of the equivalence of conditions in Proposition~\ref{eq-conditions}. Let $(A, \Phi) \in N$ be a Higgs stable smooth pair. Then the problem of finding a solution to the self-duality equations in $\mathcal{O}_{(A, \Phi)}$, as stated in Condition~(1), reduces to finding a critical point of the restricted energy function  
\[
\left.\|\mu\|_{L^2}^2\right|_{\mathcal{O}_{(A, \Phi)}} \colon \mathcal{O}_{(A, \Phi)} \to \mathbb{R}_{\geq 0},
\]
as stated in Condition~(4). In the first two subsections, if $(A, \Phi)$ lies in a dense subset of space of the smooth Higgs stable pair, we will identify the minimal points of $\left.\|\mu\|_{L^2}^2\right|_{\mathcal{O}_{(A, \Phi)}}$. In the final subsection, using the conclusions from the first two subsections, we will address the general case of any Higgs stable pair $(A, \Phi) \in N$.

\subsection{Weak Convergence Method}

We extract some useful conclusions from Hitchin's proof method as stated in Lemma~\ref{weakconvergentsubseq} and Proposition~\ref{w-converge}. Since these conclusions were not explicitly stated as propositions in the original paper, we will outline the proof given by Hitchin. 

\begin{lemma}\label{weakconvergentsubseq}
Let $\{(A_n, \Phi_n)\}$ be a sequence in $N_1^2$ such that
\[
\mathrm{d}^{\prime\prime}_{A_n} \Phi_n = 0, \quad \|F_{A_n} + [\Phi_n, \Phi_n^*]\|_{L^2} < C_1\quad\text{and}\quad\|\det \Phi_n\|_{L^2} < C_2
\]
for some constants $C_1, C_2 > 0$. Then there exists a subsequence of $\{(A_n, \Phi_n)\}$ that weakly converges in $L_1^2$ to some $(A, \Phi) \in N_1^2$ satisfying $\mathrm{d}_A^{\prime\prime} \Phi = 0$.
\end{lemma}

\begin{proof}
By~\cite[(4.6)]{hitchin1987self}, for any $(A, \Phi) \in N_1^2$ satisfying $\mathrm{d}_A^{\prime\prime}\Phi = 0$, we claim that there exists a constant $c_1 > 0$ such that
\[
\langle F_{A}, [\Phi, \Phi^*] \rangle_{L^2} + c_1 \|\Phi\|_{L^2}^2 \geq 0.
\] 
By combining the above inequality with the condition $\|F_{A_n} + [\Phi_n, \Phi_n^*]\|_{L^2} < C_1$, as asserted by~\cite[(4.9)]{hitchin1987self}, there exist constants $c_3, c_4 > 0$ such that
\begin{equation}
\|[\Phi_n, \Phi^{*}_n]\|_{L^{2}}^{2} \leq c_3 + c_4 \|\Phi_n\|_{L^2}^2, \quad \forall n \geq 1. \label{*1}
\end{equation}
By~\cite[(4.11)]{hitchin1987self}, recall that $\operatorname{Vol}(M,\omega_M) = 1$, we claim that:
\begin{equation}\label{*2}
\|[\Phi_{n}, \Phi_{n}^{*}]\|_{L^{2}}^{2} + 4 \int_{M} |\det \Phi_{n}|^{2} \geqslant \|\Phi_{n}\|_{L^{2}}^{4}.
\end{equation}
Recalling the Schwarz inequality, we know that $\|\Phi_n\|_{L^2} \leq \|\Phi_n\|_{L^4}$, and from the condition $\|\det \Phi_n\|_{L^2} < C_2$, combining Equation~\eqref{*1} and Equation~\eqref{*2}, we claim that
\[
\|\Phi_n\|_{L^2}^4 \leq c_4 \|\Phi_n\|_{L^2}^2 + (c_3 + 4 C_2),
\]
which implies that $\|\Phi_n\|_{L^2}$ is bounded. Moreover, there exist constants $K_1, K_2, K_3 > 0$ such that
\[
\|\Phi_n\|_{L^2} < K_1, \quad \|[\Phi_n, \Phi^{*}_n]\|_{L^2} < K_2, \quad \|\Phi_n\|_{L^4} < K_3.
\]
Furthermore, the term $\|F_{A_n}\|_{L^2}$ is also uniformly bounded by
\[
\|F_{A_n}\|_{L^2} < C_2 + \|[\Phi_n, \Phi^{*}_n]\|_{L^2}.
\]

By Uhlenbeck's weak compactness theorem~\cite[Theorem~1.5]{uhlenbeck1982connections}, since $\|F_{A_n}\|_{L^2}$ is bounded, there exists a subsequence of $\{A_n\}$ that weakly converges to some $A \in \mathscr{A}_1^2$. For convenience, we assume that $\{A_n\}$ converges weakly to $A$ in $L_1^2$.

Locally on a chart $(U, z)$, the pair $(A_n, \Phi_n)$ can be expressed as $\left(\mathrm{d} - \theta_n^{\dagger_H} \mathrm{d}z + \theta_n \mathrm{d}\bar{z}, \phi_n \mathrm{d}z \right)$. From the equation $\text{d}_{A_n}^{\prime\prime}\Phi_n = 0$, it follows that $\bar{\partial} \phi_n + [\theta_n \text{d}\bar{z}, \phi_n] = 0$. The sequence $\{\theta_n\}$ converges weakly in $L_1^2$ and strongly in $L^4$, since $L_1^2$ is compactly embedded in $L^4$. Combined with the boundedness of $\|\Phi_n\|_{L^4}$, we deduce that the sequence $\{\bar{\partial} \phi_n = -[ \theta_n \text{d}\bar{z}, \phi_n]\}$ is bounded in $L^2$. 

Furthermore, since both $\{\phi_n\}$ and $\{\bar{\partial} \phi_n\}$ are bounded in $L^2$, elliptic estimates imply that the sequence $\{\Phi_n\}$ is also bounded in $L_1^2$. Consequently, $\{\Phi_n\}$ admits a weakly convergent subsequence.  Moreover, $\mathrm{d}_{A_n}^{\prime\prime} \Phi_n$ weakly converges to $\mathrm{d}_A^{\prime\prime} \Phi$ in $L^2$, which implies that $\mathrm{d}_A^{\prime\prime} \Phi = 0$.
\end{proof}

\begin{prop}[Hitchin]\label{w-converge}
Let $x \in N$ be a Higgs stable smooth pair, and let $\big\{(A_i, \Phi_i)\big\}_{i \geq 1}$ be a sequence in $\mathcal{O}_x$ such that
\[
\lim_{i \to \infty} \|\mu(A_i, \Phi_i)\|_{L^2}^2 = \inf_{y \in \mathcal{O}_x} \|\mu(y)\|_{L^2}^2,
\]
where $x = (A_1, \Phi_1)$. Then, there exists a subsequence of $\big\{(A_i, \Phi_i)\big\}_{i \geq 1}$ that weakly converges in $L_1^2$ to some $(A, \Phi) \in N_1^2$ satisfying $\mathrm{d}_A^{\prime\prime} \Phi = 0$. Furthermore, there exists a non-zero $L_1^2$ section $\psi$ of $\operatorname{End} V$ such that
\[
\Phi_1 \psi = \psi \Phi \quad \text{and} \quad \mathrm{d}_{A_1}^{\prime\prime} \psi = \psi \mathrm{d}_{A}^{\prime\prime}.
\]
\end{prop}

\begin{remark}
For any Higgs stable pair $x \in N$, Proposition~\ref{w-converge} constructs an $L_1^2$-regular bundle homomorphism $\psi$. We will focus on studying this homomorphism in Proposition~\ref{generic-case}.
\end{remark}

\begin{proof}
Since $\{(A_i, \Phi_i)\}$ lies in a stable ${\mathscr{G}^\mathbb{C}}_2^2$-orbit, we have $\mathrm{d}_{A_i}^{\prime\prime} \Phi_i = 0$ and $\det \Phi_i = \det \Phi_1$, which is independent of the index $i$. Hence, $\|\det \Phi_i\|_{L^2}$ is bounded. At the same time, $(A_i, \Phi_i)$ forms a minimizing sequence of $\|\mu\|_{L^2}^2|_{\mathcal{O}_x}$. Therefore, $\|F_{A_i} + [\Phi_i, \Phi_i^*]\|_{L^2}$ is also bounded. By Lemma~\ref{weakconvergentsubseq}, we may assume that the sequence $\{(A_i, \Phi_i)\}$ weakly converges to $(A, \Phi)$ in $L_1^2$ satisfying $\mathrm{d}_A^{\prime\prime} \Phi = 0$. Next, we discuss how to find the desired $L_1^2$ bundle homomorphism $\psi$, still using the weak convergence method.

Since $(A_i, \Phi_i) \in \mathcal{O}_x$, where $x = (A_1, \Phi_1)$, there exists $g_i \in {\mathscr{G}^\mathbb{C}}^2_2$ such that $(A_i, \Phi_i) \cdot g_i = x$. We set $\psi_i$ as $g_i$ composed with some scalar transformation, such that it satisfies the following conditions:  
\begin{equation}
\mathrm{d}_{A_i A_1}^{\prime\prime} \psi_i = 0, \quad \Phi_1 \psi_i - \psi_i \Phi_i = 0, \quad \|\psi_i\|_{L^4} = C > 0.\label{*4}
\end{equation}
Here, $\mathrm{d}_{A_i A_1}^{\prime\prime}$ denotes the $(0,1)$ part of the $L_1^2$ connection $\mathrm{d}_{A_i A_1}$ on $\operatorname{End} V$, given by:  
\[
\mathrm{d}_{A_i A_1} f(s) := \mathrm{d}_{A_1}f(s) - f(\mathrm{d}_{A_i}s),
\]
for any smooth $f \colon V \to V$ and smooth section $s$ of $V$. The operator $\text{d}_{A A_1}^{\prime\prime}$ can be written as $\text{d}_{A_i A_1}^{\prime\prime} + \beta_i$ for some $\beta_i \in (L_1^2)^{0,1}(M, \operatorname{End}(V^* \otimes V))$. Since $\text{d}_{A_i A_1}^{\prime\prime} \psi_i = 0$, we have 
\[
\|\text{d}_{A A_1}^{\prime\prime} \psi_i\|_{L^2} \leq \|\text{d}_{A_i A_1}^{\prime\prime} \psi_i\|_{L^2}+\|\beta_i \psi_i\|_{L^2}=\|\beta_i \psi_i\|_{L^2} \leq \|\beta_i\|_{L^4} \|\psi_i\|_{L^4}=C\|\beta_i\|_{L^4}. 
\]
Since $A$ is the weak limit of $A_i$, it follows that $\beta_i$ converges weakly to $0$ in $L_1^2$. Consequently, $\beta_i$ converges to $0$ in $L^4$. Thus, both $\psi_i$ and $\mathrm{d}_{A A_1}^{\prime\prime} \psi_i$ are bounded in $L^2$. By applying the elliptic estimate for the operator $\mathrm{d}_{A A_1}^{\prime\prime}$, we obtain that $\|\psi_i\|_{L_1^2}$ is bounded. Therefore, we can assume that $\psi_i$ converges weakly to some $\psi$ in $L_1^2$. The limit of \eqref{*4} is
\[
\mathrm{d}_{A A_1}^{\prime\prime} \psi = 0, \quad \Phi_1 \psi - \psi \Phi = 0.
\]
Given our assumption that $\|\psi_i\|_{L^4}$ is constant and equal to $C > 0$, and $\psi_i$ converges to $\psi$ in $L^4$, it follows that $\psi \neq 0$.
\end{proof}

\subsection{Existence for the Generic Case}

Our goal is to find the minimal point of $\|\mu\|_{L^2}^2$ within $\mathcal{O}_x$ for some Higgs stable smooth pair $x$. If the weak limit $(A, \Phi)$ obtained in Proposition~\ref{w-converge} lies in $\mathcal{O}_x$, it minimizes the functional  
\[
\left.\|\mu\|_{L^2}^2\right|_{\mathcal{O}_x}: \mathcal{O}_x \to \mathbb{R}_{\geq0}.
\]  
As a result, $(A, \Phi)$ is the desired solution of class $L_1^2$ to the Hitchin self-duality equation~\eqref{hitchin-sd}. To show this, we first use the following result by Hitchin~\cite[Theorem~3.4]{hitchin1987self}  

\begin{theorem}[Hitchin]\label{irr}
Let $x = (A_x, \Phi_x)$ be a Higgs stable smooth pair. There exists a non-empty Zariski open subset $U \subseteq \operatorname{H}^0(M, \operatorname{End}_0 V \otimes K_M)$ such that any $\tilde{\Phi}_1 \in U$ leaves no holomorphic line subbundle of $(V, \operatorname{d}_{A_x}^{\prime\prime})$ invariant. 
Here, $\operatorname{H}^{0}(M, \operatorname{End}_0 V\otimes K_M)$ is the space of holomorphic sections of $\operatorname{End}_0 V \otimes K_M$ with respect to the holomorphic structure induced by $\operatorname{d}_{A_x}^{\prime\prime}$.
\end{theorem}

As a direct corollary, any pair $y := (A_y, \Phi_y)$ with $A_y = A_x$ and $\Phi_y \in U$ is naturally a smooth Higgs stable pair.

\begin{prop}\label{generic-case}
Let $x := (A_x, \Phi_x)$ be a smooth Higgs stable pair, and let $U$ be the Zariski open subset in $\operatorname{H}^{0}(M, \operatorname{End}_0 V\otimes K_M)$ described in Theorem~\ref{irr}. For any $y := (A_y, \Phi_y)$ with $A_y = A_x$ and $\Phi_y \in U$, there exists a minimal point $(A, \Phi) \in \mathcal{O}_y$ of the restricted energy function
\[
\left.\|\mu\|_{L^2}^2\right|_{\mathcal{O}_y} \colon \mathcal{O}_y \to \mathbb{R}_{\geq 0}.
\]
The minimal point $(A, \Phi) \in \mathcal{O}_y$ is a solution of $L_1^2$ class to the self-duality equation~\eqref{hitchin-sd}.
\end{prop}

\begin{proof}

Starting with $y := (A_y, \Phi_y)$, where $A_y = A_x$ and $\Phi_y \in U$, we apply the procedure outlined in Proposition~\ref{w-converge}. This produces a weak limit $(A, \Phi)$ in $L_1^2$ along with a bundle homomorphism $\psi$ satisfying:  
\[
\operatorname{d}_{A}^{\prime\prime}\Phi=0,\quad\Phi_y \psi = \psi \Phi, \quad \text{and} \quad \operatorname{d}_{A_y}^{\prime\prime} \psi = \psi \operatorname{d}_{A}^{\prime\prime}.
\]  
Here, $\psi$ is a bundle homomorphism with $L_1^2$ regularity. We only need to show $\psi \in {\mathscr{G}^\mathbb{C}}_2^2$.

First, we prove that $\psi$ has $L_2^2$ regularity. While the weak limit $(A, \Phi)$ is not necessarily smooth, Lemma~\ref{AB} guarantees the existence of a gauge transformation $g \in {\mathscr{G}^\mathbb{C}}_2^2$ such that the transformed pair $(B, \Psi) := (A, \Phi) \cdot g$ is smooth. Consequently, the following relations hold
\[
\operatorname{d}_A^{\prime\prime} g = g \operatorname{d}_B^{\prime\prime}, \quad \text{and} \quad \Phi g = g \Psi.
\]
Define $\varphi := \psi \circ g$. A direct computation shows that $\varphi$ satisfies
\begin{equation}
\operatorname{d}_{A_y}^{\prime\prime} \varphi = \varphi \operatorname{d}_B^{\prime\prime}, \quad \text{and} \quad \Phi_y \varphi = \varphi \Psi. \label{*}
\end{equation}
Lemma~\ref{AB} guarantees the smoothness of $\varphi$ since both $B$ and $A_y$ are smooth. Hence, $\psi = \varphi g^{-1}$ inherits the regularity of $g^{-1}$, elevating $\psi$ from $L_1^2$ to $L_2^2$ regularity.

Next, we prove that $\psi$ is a bundle isomorphism. Continuing with the notation introduced above, it remains to show that $\varphi := \psi g$ is a smooth bundle isomorphism. Suppose, for contradiction, that $\varphi$ is not an isomorphism. From the relations~\eqref{*}, we observe that  
\[
\varphi \colon (V, \operatorname{d}_B^{\prime\prime}) \to (V, \operatorname{d}_{A_y}^{\prime\prime})
\]
is a non-zero holomorphic homomorphism. Since $B$ and $A_y$ induce the same connection on $\det V$, denoted by $A_0$, it follows that $\det \varphi$ is a holomorphic automorphism of the line bundle $(\det V, \operatorname{d}_{A_0}^{\prime\prime})$. Thus, $\det \varphi$ must act as a nonzero scalar transformation. If $\varphi$ is not an isomorphism, it must map $V$ into a proper holomorphic subbundle $L$ of $V$ satisfying the following properties:
\begin{itemize}
    \item $L$ is a holomorphic subbundle of $(V, \operatorname{d}_{A_y}^{\prime\prime})$,
    \item $L$ is preserved under the action of $\Phi_y$.
\end{itemize}
This directly contradicts Theorem~\ref{irr}, which asserts that $(V, A_y)$ does not admit any holomorphic line subbundle invariant under $\Phi_y$. Consequently, $\varphi$ must be a bundle isomorphism. We may rescale $\varphi$ by a scalar factor so that $\det \varphi = 1$ on each fiber, ensuring that $\varphi \in {\mathscr{G}^\mathbb{C}}_2^2$. This completes the proof.
\end{proof}

\subsection{Proof of Theorem \ref{main1}}

\begin{proof}
Let $y = (A_y, \Phi_y) \in N$ be a smooth Higgs stable pair, and let $\mathcal{O}_y$ denote the ${\mathscr{G}^\mathbb{C}}_2^2$-orbit of $y$ in $N_1^2$. Based on Proposition~\ref{generic-case}, for any $(A, \Phi)$ being a Higgs stable pair with $\Phi \in U$, where $U$ is a non-empty Zariski open subset of $\operatorname{H}^{0}(M, \operatorname{End}_0 V \otimes K_M)$ as described in Proposition~\ref{generic-case}, we can obtain an $L_1^2$ solution to the self-duality equation~\eqref{hitchin-sd} on the orbit $\mathcal{O}_{(A,\Phi)}$. Consequently, we take a sequence $y_i := (A_{y_i}, \Phi_{y_i})$ with $A_{y_i} = A_y$ and $\Phi_{y_i} \in U$ such that $\{y_i\}$ converges to $y$ in $N_1^2$. Additionally, we assume that $(B_{y_i}, \Psi_{y_i}) = y_i \cdot g_i$ for some $g_i \in {\mathscr{G}^\mathbb{C}}_2^2$ is a solution in the $L_1^2$ space of the self-duality equation~\eqref{hitchin-sd}. Hence $\mathrm{d}_{B_{y_i}}^{\prime\prime} \Psi_{y_i} = 0 =  F_{B_{y_i}} + [\Psi_{y_i}, \Psi_{y_i}^*]$. Moreover, we have $\|\det \Psi_{y_i}\|_{L^2} = \|\det \Phi_{y_i}\|_{L^2}$, which remains bounded since $\Phi_{y_i}$ converges to $\Phi_y$ in $L_1^2$. By Lemma~\ref{weakconvergentsubseq}, we may assume that $(B_{y_i}, \Psi_{y_i})$ weakly converges to $(B, \Psi)$ in $L_1^2$. Furthermore, the self-duality equation is preserved under the weak convergence process, i.e., $(B, \Psi)$ satisfies the Hitchin self-duality equation~\eqref{hitchin-sd} of class $L_1^2$.

Note that $(B_{y_i}, \Psi_{y_i}) = (A_{y_i}, \Phi_{y_i}) \cdot g_i$. Then, there exists $\psi_i$, which differs from $g_i$ by a scalar transformation, such that
\[
\operatorname{d}^{\prime\prime}_{B_{y_i}} \psi_i = \psi_i \operatorname{d}^{\prime\prime}_{A_{y_i}}, \quad 
\Psi_{y_i} \psi_i = \psi_i \Phi_{y_i}, \quad 
\|\psi_i\|_{L^4} = C > 0.
\]
Since $A_{y_i} = A_y$ and $\Phi_{y_i}$ converges to $\Phi_y$ in $L_1^2$, we have  
\[
\operatorname{d}^{\prime\prime}_{B_{y_i}} \psi_i = \psi_i \operatorname{d}^{\prime\prime}_{A_y} , \quad 
\Psi_{y_i} \psi_i- \psi_i \Phi_y \xrightarrow{L^2}0.
\]
Since $B_{y_i}$ converges weakly to $B$ in $L_1^2$, it follows that $\operatorname{d}^{\prime\prime}_B \psi_i - \psi_i \operatorname{d}^{\prime\prime}_{A_y}\to 0$ in $L^2$. Hence, both $\psi_i$ and $\operatorname{d}^{\prime\prime}_{B A_y} \psi_i$ are bounded in $L^2$, which implies that $\{\psi_i\}$ is bounded in $L_1^2$. By the weak compactness of the $L_1^2$ space, we can assume that $\psi_i$ converges weakly to some $L_1^2$ bundle homomorphism $\psi$ of $V$. Since $B_{y_i}$ and $\Psi_{y_i}$ converge weakly to $(B, \Psi)$ in $L_1^2$, we have constructed a non-zero $L_1^2$ section $\psi$ of $\operatorname{End} V$ such that
\[
\operatorname{d}^{\prime\prime}_B \psi = \psi \operatorname{d}^{\prime\prime}_{A_y} , \quad 
\Psi \psi = \psi \Phi_y.
\]

The pair $(A_y, \Phi_y)$ is a Higgs stable pair, and $(B, \Psi)$ is a solution of the self-duality equation. By~\cite[Theorem~2.1]{hitchin1987self}, $(B, \Psi)$ is Higgs semi-stable. By Lemma~\ref{Hitchin315}, we conclude that $\psi$ is an $L_2^2$ bundle isomorphism. Moreover, through an appropriate scalar transformation, we can assume $\det \psi = 1$, i.e., $\psi \in {\mathscr{G}^\mathbb{C}}_2^2$. Therefore, $(B, \Psi)$ is a solution to the self-duality equation that lies in the orbit $\mathcal{O}_y$. This completes the proof of Theorem~\ref{main1}.
\end{proof}

\section{Smooth Pairs in the $\mathscr{G}_2^2$-Orbits}

In this section, we focus on proving Theorem~\ref{main2}. Let $(A, \Phi) \in N_1^2$ be a solution lie in $L_1^2$ space of the self-duality equation~\eqref{hitchin-sd}. Our objective is to construct a smooth pair in the $\mathscr{G}_2^2$ orbit of $(A, \Phi)$.

Since the solutions to the self-duality equation~\eqref{hitchin-sd} are obtained through a weak convergence process in the previous section, they are not necessarily smooth. It is therefore crucial to establish the existence of smooth solutions. First, we recall a key result from K. K. Uhlenbeck~\cite[Theorem~1.3]{uhlenbeck1982connections} and Katrin Wehrheim~\cite[Theorem~B]{wehrheim2004uhlenbeck}.

\begin{theorem}[Coulomb Gauge]\label{U-reg}
Let $\mathbb{B}^2$ be a disk in $\mathbb{C}$, and let $B := \mathbb{B}^2 \times \mathbb{C}^2$ be a trivial complex vector bundle equipped with a Hermitian metric $H$. Let $A = \text{d} + \theta$ be an $L_1^2$ unitary connection on $(B, H)$. Then there exist constants $C > 0$ and $\varepsilon > 0$ such that the following holds: If $\|F_A\|_{L^2} < \varepsilon$, then there exists a unitary gauge transformation $g \in \mathscr{G}_2^2$ such that $\tilde{A} := g^{-1} A g = \text{d} + \tilde{\theta}$ satisfies:
\[
\operatorname{d}^* \tilde{\theta} = 0, \quad *\tilde{A}|_{\partial \mathbb{B}^2} = 0, \quad \|\tilde{\theta}\|_{L_1^2} \leq C \|F_A\|_{L^2}.
\]
\end{theorem}

\subsection{Proof of Theorem~\ref{main2}}
\begin{proof}
Let $(A, \Phi)$ be an $L_1^2$ solution to \eqref{hitchin-sd} then $\|F_A\|_{L^2} = \|[\Phi, \Phi^*]\|_{L^2} < \infty.$ Take a local chart $(U, z)$ with a given local frame of $V$ such that $\|F_A|_U\|_{L^2}$ is sufficiently small. We express $A = \mathrm{d} + \theta, ~ \Phi = \phi \, \mathrm{d}z,$ where $\theta$ is a local skew-Hermitian 1-form and $\phi$ is a locally matrix-valued function. The self-duality equations can then be written locally as:
\begin{align*}
&\text{d}\theta + \theta \wedge \theta + [\phi, \phi^{\dagger_H}] \, \text{d}z \wedge \text{d}\bar{z} = -\pi i \deg(V) \, \text{id}_V \, \omega_M, \\
&\bar{\partial} \phi + [\theta^{0,1}, \phi] = 0.
\end{align*}

Consider the $L_2^2$ unitary gauge transformation $g$ on $V|_U$ provided by Coulomb gauge Theorem \ref{U-reg}, and let $(\tilde{A} = \text{d} + \tilde{\theta}, \tilde{\Phi} = \tilde{\phi} \, \text{d}z) := (A, \Phi) \cdot g$. This gauge-transformed pair naturally satisfies the self-duality equation \eqref{hitchin-sd}, leading to the following coupled system of elliptic equations:
\begin{align*}
&\Delta \tilde{\theta} + \text{d}^* \Big(\tilde{\theta} \wedge \tilde{\theta} + [\tilde{\phi}, \tilde{\phi}^{\dagger_H}] \, \text{d}z \wedge \text{d}\bar{z}\Big) = 0, \\
&\bar{\partial} \tilde{\phi} + [\tilde{\theta}^{0,1}, \tilde{\phi}] = 0.
\end{align*}
By Corollary~\ref{SobM}, since $\bar{\partial} \tilde{\phi} = -[\tilde{\theta}^{0,1}, \tilde{\phi}]$ has $L_1^{2 - \varepsilon}$ regularity, applying elliptic estimates again implies $L_2^{2 - \varepsilon}$ regularity for $\tilde{\phi}$. Similarly, $\Delta \tilde{\theta} = -*\mathrm{d}*\Big(\tilde{\theta} \wedge \tilde{\theta} + [\tilde{\phi}, \tilde{\phi}^{\dagger_H}] \, \mathrm{d}z \wedge \mathrm{d}\bar{z}\Big)$ has $L^{2 - \varepsilon}$ regularity. Using elliptic estimates, it follows that $\tilde{\theta}$ has $L_2^{2 - \varepsilon}$ regularity. Repeating this iterative process successively improves the regularity of $\tilde{\theta}$ and $\tilde{\phi}$. Thus, we conclude that $(\tilde{A}, \tilde{\Phi})$ is smooth in $U$.

Let $\big\{(U_\alpha, z_\alpha)\big\}_{\alpha=1}^n$ be an atlas of a Riemann surface $M$, such that each $U_\alpha$ is diffeomorphic to a disk in $\mathbb{C}$.  Let $V_\alpha := V\big|_{U_\alpha}$, and suppose there exists an $L_2^2$ unitary gauge transformation $g_\alpha$ of $V_\alpha$ such that the pair $(A_\alpha, \Phi_\alpha) := (A, \Phi)\big|_{U_\alpha} \cdot g_\alpha$ is a smooth pair on $U_\alpha$. Since the space of smooth unitary gauge transformations is dense in the space of $L_2^2$ unitary gauge transformations on $(V_\alpha, H)$, there exists a smooth unitary gauge transformation $s_\alpha$ on $V_\alpha$ such that $\|g_\alpha^{-1} - s_\alpha\|_{L_2^2} < \varepsilon_1$ for any fixed small positive number $\varepsilon_1$. Let $\mathrm{id}_\alpha$ denote the identity map on $V_\alpha$. Then,
\[
\|\mathrm{id}_\alpha - g_\alpha s_\alpha\|_{L_2^2} \leq C_\alpha \|g_\alpha\|_{L_2^2} \left\|g_\alpha^{-1} - s_\alpha\right\|_{L_2^2},
\]
where $C_\alpha$ depends only on $U_\alpha$. Hence, we can make an assumption that
\[
\|g_\alpha - \mathrm{id}_\alpha\|_{L_2^2} < \varepsilon_2 \quad \text{for each } \alpha = 1, \cdots, n,
\]
where $\varepsilon_2$ is a fixed small positive number. Denote $V_{\alpha\beta} := V|_{U_\alpha \cap U_\beta}$ for each $U_\alpha \cap U_\beta \neq \emptyset$. Let $g_{\alpha\beta} := g_\beta^{-1} \cdot g_\alpha$, then $g_{\alpha\beta}$ is a gauge transformation on $V_{\alpha\beta}$ that sends the smooth pair $(A_\beta, \Phi_\beta)|_{U_{\alpha\beta}}$ to another smooth pair $(A_\alpha, \Phi_\alpha)|_{U_{\alpha\beta}}$. By Theorem \ref{A-B}, $g_{\alpha\beta}$ is smooth. Moreover, we can assume that
\[
\|g_{\alpha\beta} - \mathrm{id}_{\alpha\beta}\|_{L_2^2} < \varepsilon_3,
\]
where $\varepsilon_3$ is a fixed small positive number. There is an open neighborhood of $0 \in \mathfrak{su}_2$ such that $\exp : \mathfrak{su}_2 \to SU(2)$ is a diffeomorphism on it. When we use the symbols $\exp \xi$ or $\exp^{-1}s$, we mean that $\xi$ is near $0 \in \mathfrak{su}_2$ or $s$ is near $\mathrm{id} \in SU(2)$, respectively. In this case, we can suppose that $g_{\alpha\beta} = \exp \eta_{\alpha\beta}$, where $\eta_{\alpha\beta}$ is a smooth function on $V_{\alpha\beta}$ with values near $0\in\mathfrak{su}_2$. Consider the distance function $\text{d}: M \times M \to \mathbb{R}$ given by the Kähler metric on $M$, and define 
\[
U_\alpha^\delta = \big\{z \in U_\alpha \ | \ \text{d}(z, \partial U_\alpha) < \delta\big\}.
\]
For $\delta$ sufficiently small, we can assume that $\{U_\alpha^{n\cdot\delta}\}_{\alpha=1}^n$ is also an open cover of $M$. 

Next, we will construct a global unitary gauge transformation $\tilde{g}$ such that $g_\alpha^{-1} \cdot \tilde{g}\big|_{U_\alpha^{n\cdot\delta}}$ is smooth on each $U_\alpha^{n\cdot\delta}$ for $1 \leq \alpha \leq n$. Consequently, $(A, \Phi) \cdot \tilde{g}$ is the smooth pair we want, since
\[
\big(A, \Phi\big) \cdot \tilde{g} = (A_\alpha, \Phi_\alpha) \cdot \big(g_\alpha^{-1} \cdot \tilde{g}\big)
\]
is smooth on each $U_\alpha^{n\cdot\delta}$. We use induction to prove the following proposition: For any $1 \leq k \leq n$, there exists a sequence of local unitary gauge transformations 
\[
\Big\{\tilde{g}_\alpha^{(k)}: V\big|_{U_\alpha^{k\cdot\delta}} \to V\big|_{U_\alpha^{k\cdot\delta}}\Big\}_{\alpha=1}^k
\]
such that: $g_\alpha^{-1} \tilde{g}_\alpha^{(k)}$ is smooth, and $\tilde{g}_\alpha^{(k)} = \tilde{g}_\beta^{(k)}$ if $U_\alpha^{k\cdot\delta} \cap U_\beta^{k\cdot\delta} \neq \emptyset$.

For $k = 1$, define $\tilde{g}_1^{(1)} = g_1^{(1)}\big|_{U_1^\delta}$. Clearly, $\tilde{g}_1^{(1)}$ satisfies the required conditions. Assume the proposition holds for $k$. Next, we prove the statement for $k+1$. Denote $\mathcal{U}_k^\varepsilon := \bigcup_{1 \leq \alpha \leq k} U_\alpha^\varepsilon$ and consider the two following cases

Case 1: If $U_{k+1}^{(k+1)\cdot\delta} \cap \mathcal{U}_k^{(k+1)\cdot\delta} = \emptyset$, define $\tilde{g}_\alpha^{(k+1)} := \tilde{g}_\alpha^{(k)}\big|_{U_\alpha^{(k+1)\cdot\delta}}$ for $1 \leq \alpha \leq k$ and $\tilde{g}_{k+1}^{(k+1)} := g_{k+1}\big|_{U_{k+1}^{(k+1)\cdot\delta}}$. Clearly, this choice satisfies the required conditions.

Case 2: If $U_{k+1}^{(k+1)\cdot\delta} \cap \mathcal{U}_k^{(k+1)\delta} \neq \emptyset$, there exists a smooth function $\phi_{k+1}: U_{k+1}^{(k+1)\cdot\delta} \to [0, 1]$ such that:
\[
\phi_{k+1} =
\begin{cases} 
0, & \text{on } U_{k+1}^{(k+1)\cdot\delta} \cap \mathcal{U}_k^{(k+1)\cdot\delta}, \\ 
1, & \text{on } U_{k+1}^{(k+1)\cdot\delta} \setminus \mathcal{U}_k^{k\cdot\delta}.
\end{cases}
\]
We define $\tilde{g}_\alpha^{(k+1)} := \tilde{g}_\alpha^{(k)}\big|_{U_\alpha^{(k+1)\cdot\delta}}$ for $1 \leq \alpha \leq k$, and
\[
\tilde{g}_{k+1}^{(k+1)} = 
\begin{cases}
\tilde{g}_\alpha^{(k)} \exp\left\{\phi_{k+1} \exp^{-1} \left\{\tilde{g}_\alpha^{(k)-1} g_{k+1}\right\}\right\}, & \quad \text{on } U_{k+1}^{(k+1)\cdot\delta} \cap U_\alpha^{k\cdot\delta}, \\ 
g_{k+1}, & \quad \text{on } U_{k+1}^{(k+1)\cdot\delta} \setminus \mathcal{U}_k^{k\cdot\delta}.
\end{cases}
\]
Since $\tilde{g}_\alpha^{(k)} = \tilde{g}_\beta^{(k)}$ whenever $U_\alpha^{k\cdot\delta} \cap U_\beta^{k\cdot\delta} \neq \emptyset$, this definition is well-defined. First, $\tilde{g}_{k+1}^{(k+1)} = \tilde{g}_\alpha^{(k+1)}$ on $U_{k+1}^{(k+1)\cdot\delta} \cap U_\alpha^{k\cdot\delta} \neq \emptyset$ since $\phi_{k+1} = 0$. Second, $g_{k+1}^{-1} \tilde{g}_{k+1}^{(k+1)}$ is smooth, as required, since
\begin{align*}
g_{k+1}^{-1} \tilde{g}_{k+1}^{(k+1)} &= g_{k+1}^{-1} \tilde{g}_\alpha^{(k)} \exp\left\{\phi_{k+1} \exp^{-1} \left\{\tilde{g}_\alpha^{(k)-1} g_{k+1}\right\}\right\} \\ 
&= \left(g_{k+1}^{-1} g_\alpha\right)\left(g_\alpha^{-1} \tilde{g}_\alpha^{(k)}\right) \exp\Big\{\phi_{k+1} \exp^{-1}\Big\{\left(g_\alpha^{-1} \tilde{g}_\alpha^{(k)}\right)^{-1} \left(g_\alpha^{-1} g_{k+1}\right)\Big\}\Big\}.
\end{align*}
We emphasize that all components in this expression, namely $g_{k+1}^{-1} g_\alpha$, $g_\alpha^{-1} \tilde{g}_\alpha^{(k)}$, $\exp(\cdot)$, and $\phi_{k+1}$, are smooth. 
\end{proof}

\section{Conclusion}

Let $x \in N$ be a Higgs stable pair, and let $\mathcal{O}_x$ denote the ${\mathscr{G}^\mathbb{C}}_2^2$-orbit of $x$ in $N_1^2$. Following the framework of Hitchin's existence method, we found $x_1 \in \mathcal{O}_x$, which is an $L_1^2$ solution of the self-duality equations~\eqref{hitchin-sd}. The main innovation of this note lies in utilizing the Coulomb gauge to find a smooth pair $x_2 \in N$ within the ${\mathscr{G}_2^2}$-orbit of $x_1$, which is a smooth solution of the self-duality equations. Finally, we have rigorously confirmed the existence of smooth solutions within $\mathcal{O}_x$, as originally demonstrated by Hitchin.\\

\noindent\textbf{Acknowledgement:}
I would like to thank Dr. Nianzi Li for his detailed suggestions on the article, pointing out the sign error I made while computing the local expression of $\text{d}_{A}^{\prime\prime}\Phi$. I also thank my classmate Yuanjiu Lv for our discussions, especially for the inspiration he gave me in the key step of obtaining the global gauge transformation from the local Coulomb gauge in Section 5.

\bibliographystyle{plain}
\bibliography{Ref}

\end{document}